\documentclass[11pt]{amsart}
\usepackage[margin=1in]{geometry}
\usepackage{amsmath,amssymb,amscd,mathtools}
\usepackage{microtype}
\usepackage[colorlinks=false,pdfborder={0 0 0}]{hyperref}

\newcommand{\CP}{\mathbb{CP}}
\newcommand{\Z}{\mathbb Z}
\newcommand{\R}{\mathbb R}
\newcommand{\F}{\mathbb F}
\newcommand{\RP}{\mathbb{RP}}

\newcommand{\Fix}{\operatorname{Fix}}

\newcommand{\divis}{\operatorname{div}}

\title[Real Albanese map and smooth topology]
{The Real Albanese Map and Smooth Topology of the Cartwright--Steger Quotient}
\author{Anar Akhmedov}
\address{School of Mathematics, University of Minnesota, Minneapolis, MN 55455, USA}
\email{akhmedov@math.umn.edu}
\author{Sai-Kee Yeung}
\address{Department of Mathematics, Purdue University, West Lafayette, IN 47907, USA}
\email{yeung@math.purdue.edu}
\date{September 2026}
\subjclass[2020]{Primary 14P25, 57R55; Secondary 14J29, 32Q45, 57R40}
\keywords{Cartwright--Steger surface, Albanese map, real surface, circle surgery, branched covering, smooth four-manifold}

\theoremstyle{plain}
\newtheorem{theorem}{Theorem}
\newtheorem{proposition}{Proposition}[section]
\newtheorem{lemma}[proposition]{Lemma}
\newtheorem{corollary}[proposition]{Corollary}
\theoremstyle{definition}
\newtheorem{question}[proposition]{Question}

\begin{document}
\begin{abstract}
Let $X$ be the Cartwright--Steger surface with its real involution $c$ and
Albanese map $\alpha:X\to E$, and put $Y=X/\langle c\rangle$.  Every smooth
real Albanese fiber has two ovals and is nondividing.  The branch surface
$B\cong\#_3\mathbb{RP}^2$ satisfies
$j_*\pi_1(B)=4\mathbb Z\subset\pi_1(Y)\cong\mathbb Z$ and is nullhomologous
modulo two.  Every smoothly embedded $+1$-sphere representing the positive
generator of $H_2(Y;\mathbb Z)$ meets $B$ in at least six points.  Its lift
has square two, pairs trivially with $K_X$, and has no connected embedded
representative of genus zero or one.  Circle surgery on an Albanese generator
is homeomorphic to $\mathbb{CP}^2$ for either normal framing.  Blowing down a
positive sphere, if one exists, gives a $\mathbb Z[\mathbb Z]$-homology
$S^1\times S^3$ homotopy equivalent to $S^1\times S^3$.
\end{abstract}

\maketitle

\section{Introduction}

Let $X$ be the Cartwright--Steger surface, let
\(\alpha:X\to E\) be its Albanese map, and let $c$ be the real involution
coming from the bicanonical model over $\mathbb Q$.  We study simultaneously
the restriction
\[
 \alpha_{\mathbb R}:X(\mathbb R)\longrightarrow E(\mathbb R)
\]
and the smooth topology of the quotient
\[
 Y=X/\langle c\rangle.
\]
These two problems are linked by the descended Albanese map to the M\"obius
band $E/\langle c_E\rangle$.  Its boundary contains the image of the real
locus, while its core detects the generator of $\pi_1(Y)$ used in circle
surgery.

Cartwright, Koziarz, and Yeung determined the Albanese fibration, the
order-three symmetry, the fundamental group, and the integral first homology
of $X$ \cite{CKY}.  Koziarz and Yeung proved that the fibration is stable
\cite{KY}.  Borisov and Yeung constructed the bicanonical model over
$\mathbb Q$ \cite{BY}, and Rito computed equations for the three invariant
Albanese fibers \cite{Rito}.  Stipsicz and Szab\'o proved that
\[
 X(\mathbb R)\cong\#_3\mathbb{RP}^2,
 \qquad
 Y\approx\mathbb{CP}^2\#(S^1\times S^3),
\]
that $\pi_1(Y)\cong\mathbb Z$, and that $\alpha_{\mathbb R}$ has exactly one
critical point, a saddle \cite{SS}.  Their work leaves open the smooth type
of $Y$ and does not determine the global topology of the real Albanese map.

The main theorem combines an exact calculation in the Cartwright--Steger
lattice with branched-cover topology, gauge theory, and circle surgery.

\begin{theorem}\label{thm:main}
Let $B\subset Y$ be the image of $X(\mathbb R)$ and let
$j:B\hookrightarrow Y$ denote the inclusion.
\begin{enumerate}
\item Every smooth real Albanese fiber has exactly two real components and is
nondividing.  The Reeb map of $\alpha_{\mathbb R}$ has degree two, and
\[
 j_*\pi_1(B)=4\mathbb Z\subset\pi_1(Y)\cong\mathbb Z.
\]
The quotient of every smooth real fiber by conjugation is
$N_{18}\setminus(D^2\sqcup D^2)$.
\item The branch surface satisfies
\([B]=0\) in $H_2(Y;\mathbb Z/2)$.  Every smoothly embedded $+1$-sphere
representing the positive generator of $H_2(Y;\mathbb Z)$ meets $B$, after
transverse perturbation, in at least six points.  Its lifted class $A$ has
\[
 A^2=2,\qquad K_X\cdot A=0.
\]
Every connected embedded oriented representative of $A$ has genus at least
two, and $\operatorname{PD}(A)$ is not of type $(1,1)$.
\item Circle surgery on an Albanese generator gives a manifold
orientation-preservingly homeomorphic to $\mathbb{CP}^2$ for either normal
framing.  The surgery has a natural equivariant lift to $X$.
\item If the positive generator is represented by a smooth $+1$-sphere, its
blowdown $Q$ satisfies $Q\simeq S^1\times S^3$ and has the
$\mathbb Z[\mathbb Z]$-homology of $S^1\times S^3$.
\end{enumerate}
\end{theorem}

For part (1), if a smooth real fiber has $k$ ovals, the
cut cobordism and its Reeb graph show that $k$ is the divisibility of the
restricted Albanese class.  We express this divisibility as an index in the
fixed subgroup of the Cartwright--Steger lattice.  An exact finite-quotient
calculation gives $k=2$.  The invariant genus-$19$ fiber and its genus-six
cyclic quotient then show that every smooth real fiber is nondividing.  The
same calculation determines the subgroup of $\pi_1(Y)$ carried by the branch
surface.

For the second part, a positive sphere lifts to a connected branched double
cover.  The Seiberg--Witten sphere obstruction and adjunction force the lift
to have genus at least two, which yields the six-point bound.  The real
involution and the Hodge index theorem give the remaining restrictions on
the lifted class.  Circle surgery uses a loop mapping with degree one to the
core of the quotient M\"obius band.  The blowdown under the sphere hypothesis is analyzed by
the homology of its infinite cyclic cover.

The bicanonical equations also give a base-point-free weight-two subsystem on
the invariant fiber.  It descends to an explicit degree-$24$ map from the
genus-six quotient to $\mathbb P^1$, with ramification divisor of degree
$58$.  We include this calculation because it gives an algebraic model for
the real invariant fiber and may be useful in an explicit monodromy calculation.

Akhmedov's earlier work, including joint work with B. D. Park, used
symplectic sums, torus surgery, and knot surgery to construct irreducible
exotic smooth structures on a broad range of small simply connected
four-manifolds \cite{AkhSmall,APsmall,APodd}.  Akhmedov, Sakall\i, and Yeung
used the Cartwright--Steger surface and related ball quotients as building
blocks for symplectic sums \cite{ASY}.  The operation here is instead circle
surgery on the real quotient.  We do not claim an exotic smooth structure
on $\mathbb{CP}^2$, and the existence of a positive sphere in $Y$ remains
open.

Nevertheless, the structure described below leads us to expect that the
positive generator of $H_2(Y;\mathbb Z)$ may be represented by a smooth
$+1$-sphere.  Constructing such a sphere and identifying its blowdown with
$S^1\times S^3$ would provide a direct route to proving that
\[
Y\cong\mathbb{CP}^2\#(S^1\times S^3).
\]

Sections~\ref{sec:real-map-start}--\ref{sec:weight-two} determine the real
Albanese map and the invariant fiber.  The circle-surgery, branch-surface,
and blowdown sections prove the remaining parts of
Theorem~\ref{thm:main}.  The final section records the two smooth questions
that remain.  Appendix~\ref{app:certificates} describes the exact
certificates and the accompanying verification files.

\section{The real Cartwright--Steger quotient}

Let $X$ be the Cartwright--Steger surface.  It is a minimal surface of general
type \cite{CKY} with
\[
 e(X)=3,\qquad \sigma(X)=1,\qquad b_1(X)=2.
\]
Its Albanese fibration has generic fiber of genus $19$, all fibers are
reduced, and the fibration is stable \cite{CKY,KY}.  Borisov and Yeung gave
bicanonical equations over $\mathbb Q$, so coordinatewise conjugation defines
an antiholomorphic involution \cite{BY}.  Stipsicz and Szab\'o identify this
real structure, up to composition with the holomorphic $\mathbb Z/3$-action,
with a reflection in their $S_3$-action \cite[Corollary~2.5]{SS}.  We fix the
reflection represented by the matrix $K$ in
\eqref{eq:reflection-matrix} and denote it by $c$.  If coordinate
conjugation is $\sigma^j c$, where $\sigma$ generates the order-three
action, then $\sigma$ fixes the Albanese point $O$ and the three coordinate
points $[U_7],[U_8],[U_9]$ used below.  Consequently the invariant fiber and
these three real points are the same for $c$ and for coordinate conjugation.
This also fixes the involution on $\Pi$ used in the computation.  Write
\[
 q_X:X\longrightarrow Y=X/\langle c\rangle,\qquad
 R=X^c,\qquad B=q_X(R).
\]
The restriction of $q_X$ identifies $R$ with the branch surface $B\subset Y$.
Since $c$ is antiholomorphic in complex dimension two, it preserves the real
orientation.  Equivariant smooth linearization near a point of $R$
\cite[Chapter~VI, Section~2]{Bredon} gives
real coordinates $(x,z)\in\mathbb R^2\times\mathbb C$ in which
\[
 c(x,z)=(x,-z),\qquad R=\{z=0\}.
\]
Choose the local orientation so that the complex orientation upstairs and
the quotient orientation downstairs agree away from the fixed set.  In the
corresponding oriented smooth quotient coordinates the map is
$(x,z)\mapsto(x,z^2)$.  These are smooth real coordinates, not holomorphic
coordinates for the antiholomorphic involution.  In particular, $Y$ is smooth
along $B$.  For the nonorientable surface $B$, write $e_N(B)$ for the Euler
number of its normal bundle evaluated on the twisted fundamental class of
$B$.  Local normal orientations need not agree globally; their changes are
absorbed by the orientation local system.  We fix the sign convention by
\[
 \sigma(X)=2\sigma(Y)-\frac12e_N(B).
\]
Here $\simeq$, $\approx$, and $\cong$ denote oriented homotopy equivalence,
orientation-preserving homeomorphism, and orientation-preserving
diffeomorphism.

We use the following results of Stipsicz and Szab\'o \cite{SS} concerning the
topology of the real quotient.

\begin{proposition}\label{prop:input}
The quotient $Y$ is a smooth oriented four-manifold, and $q_X$ is a
two-fold branched covering with branch surface
$B\cong\#_3\mathbb{RP}^2$.  Moreover,
\[
 \pi_1(Y)\cong\mathbb Z,\qquad
 Y\approx\CP^2\#(S^1\times S^3).
\]
The Albanese map descends to a map
$\overline\alpha:Y\to\mathcal M$, where $\mathcal M$ is a M\"obius band, and
\[
 \overline\alpha_*:\pi_1(Y)\longrightarrow\pi_1(\mathcal M)
\]
is an isomorphism.  If $\Lambda=\mathbb Z[t,t^{-1}]$ is defined using a
generator of $\pi_1(Y)$, then
\[
 H_2(Y;\Lambda)\cong\Lambda,
\]
and the equivariant intersection form is $\langle1\rangle$.
\end{proposition}

\begin{proof}
The descended Albanese map and the induced fundamental-group isomorphism are
given by \cite[Theorem~1.1(a)]{SS}; the quotient and branch surface are given
by \cite[Corollary~1.2]{SS}.  The identification of
the real involution with the reflection quotient is
\cite[Corollary~2.5]{SS}, and the equivariant intersection form is
\cite[Corollary~4.13]{SS}.
\end{proof}

These results describe the topological quotient.  The arguments below
determine further properties of the real Albanese map, the embedding of the
branch surface, positive spheres, and circle surgery.

In particular,
\[
 e(Y)=1,\qquad \sigma(Y)=1,\qquad
 b_1(Y)=1,\quad b_2^+(Y)=1,\quad b_2^-(Y)=0.
\]
Choose the positive generator $u\in H^2(Y;\mathbb Z)$, so $u^2=1$.

The branch surface maps to the boundary of the M\"obius band.  Indeed, if
$x\in R=X^c$, equivariance of the Albanese map gives
$\alpha(x)=\alpha(c(x))=c_E(\alpha(x))$.  Thus
$\alpha(x)\in E^{c_E}$, whose image under $q_E:E\to E/\langle c_E\rangle
=\mathcal M$ is $\partial\mathcal M$.  Consequently
$\overline\alpha(B)\subset\partial\mathcal M$.

We begin with circle surgery.  Orient a core
$\delta\subset\operatorname{int}\mathcal M$ disjoint from the critical values
of $\overline\alpha$.
The surface bundle over $\delta$ has a section: after writing it as the
mapping torus of its monodromy, any path from a point of the fiber to its
image closes to a section.  Its image gives an embedded
loop $\gamma\subset Y$ for which
$\overline\alpha|_\gamma$ has degree one.  We call $\gamma$ an
\emph{Albanese generator}.  Since the branch surface maps to
$\partial\mathcal M$, this section lies in $Y\setminus B$.

\begin{theorem}\label{thm:circle}
For either normal framing, circle surgery on an Albanese generator gives a
smooth manifold $Y_\gamma$ orientation-preservingly homeomorphic to
$\CP^2$.  The inverse image $\widetilde\gamma=q_X^{-1}(\gamma)$ is connected,
and the two framing classes downstairs pull back to the same framing class.
Surgery on $\widetilde\gamma$ gives a double branched cover
\[
 \widetilde P\longrightarrow Y_\gamma
\]
whose branch locus is $B\sqcup S$, where $S$ is the dual sphere.  Moreover,
\[
 e(\widetilde P)=5,\quad \sigma(\widetilde P)=1,\quad
 H_1(\widetilde P)\cong H_3(\widetilde P)\cong\mathbb Z,
\]
and
\[
 H_2(\widetilde P)\cong\mathbb Z^5,\qquad
 b_2^+(\widetilde P)=3,\quad b_2^-(\widetilde P)=2.
\]
\end{theorem}

The branch surface gives a stronger restriction on the positive generator.

\begin{theorem}\label{thm:branch}
One has $[B]=0$ in $H_2(Y;\mathbb Z/2)$.  If a smoothly embedded
$+1$-sphere $E$ represents the positive generator, then every transverse
perturbation of $E$ meets $B$ in at least six points.  Put
\[
 a=q_X^*u\in H^2(X;\mathbb Z),\qquad
 A=\operatorname{PD}_X^{-1}(a).
\]
Then
\[
 A^2=2,\qquad H_2(X;\mathbb Q)^{c_*}=\mathbb Q A,\qquad K_X\cdot A=0.
\]
Every connected embedded oriented representative of $A$ has genus at least
two, and $a$ is not of type $(1,1)$.
\end{theorem}

Finally, a sphere representing the positive generator would give a
constrained blowdown.

\begin{theorem}\label{thm:blowdown}
If the positive generator is represented by a smooth $+1$-sphere $E$, its
blowdown $Q_E$ satisfies
\[
 \pi_1(Q_E)\cong\mathbb Z,\qquad
 H_0(\widetilde Q_E)\cong H_3(\widetilde Q_E)
       \cong\Lambda/(t-1),
 \qquad H_i(\widetilde Q_E)=0\quad(i\ne0,3).
\]
Consequently,
\[
 Q_E\simeq S^1\times S^3,\qquad Y\cong Q_E\#\CP^2.
\]
\end{theorem}

Proposition~\ref{prop:input} supplies the topological quotient data.  The real Albanese calculation below strengthens this input by determining the oval number and the subgroup represented by the branch surface.  The later sections use these data to study circle surgery, positive spheres, and the result of blowing down a positive sphere.

The real involution is induced on the complex ball by the antiunitary
reflection $s(z)=K\overline z$, where
\begin{equation}\label{eq:reflection-matrix}
 K=\begin{pmatrix}
 -1&-1+\zeta_{12}+\zeta_{12}^2-\zeta_{12}^3&0\\
 0&\zeta_{12}^3&0\\
 0&0&-1
 \end{pmatrix}.
\end{equation}

We use the integral homology basis and period lattice of
\cite{CKY}; in this normalization $\alpha_*:H_1(X;\mathbb Z)\to
H_1(E;\mathbb Z)$ is an isomorphism.  We also use the smoothness of the
three invariant fibers \cite{KY,Rito} and the fact that every Albanese fiber
is irreducible and reduced \cite[Section~1.6]{Keum}.

\section{The cut cobordism and Reeb graph}\label{sec:real-map-start}

The quoted fixed-set calculation does not determine the cut cobordism or
the Reeb graph.  Both are determined by the single saddle.
We record a general statement for circle-valued Morse functions on
$N_3=\#_3\RP^2$.

\begin{theorem}\label{thm:reeb}
Let $f:N_3\to S^1$ be a surjective Morse map with exactly one critical
point, of index one.  Let $t_0$ be a regular value and put
$k=|\pi_0(f^{-1}(t_0))|$.  Then $k>0$ and:
\begin{enumerate}
\item the unique saddle is the nonorientable self-attachment of a
two-dimensional $1$-handle and preserves the number of components of a
regular level;
\item cutting $N_3$ along a regular level gives
\[
 \left(\RP^2\setminus(D^2\sqcup D^2)\right)
 \sqcup\coprod_{j=1}^{k-1}(S^1\times[0,1]);
\]
\item the boundary-gluing identification cyclically permutes the $k$
components;
\item the Reeb graph $\mathcal R_f$ is a circle, and the induced map
$\mathcal R_f\to S^1$ is a topological covering of degree $k$.
\end{enumerate}
If $\eta$ generates $H^1(S^1;\Z)$, then
\[
                         f^*\eta=k\beta
\]
for a primitive class $\beta\in H^1(N_3;\Z)$.  Thus $k$ is exactly the
divisibility of $f^*\eta$.
\end{theorem}

\begin{proof}
Cut the target at a regular value disjoint from the critical value and let
$W$ be its inverse image.  Both horizontal boundaries of $W$ have $k$
circles.  There are no components disjoint from the horizontal boundary,
since such a component would force an additional local extremum.  Away from
the critical level, Ehresmann's theorem gives product collars, and the unique
critical point adds one relative $1$-handle.  If its feet lie on distinct
level circles, the saddle merges two components; if an untwisted handle has
both feet on one circle, it splits that circle.  Either operation changes the
number of outgoing circles by one.  Since the incoming and outgoing ends both
have $k$ components, the handle has both feet on the same circle and is
twisted.  Thus every component not containing the critical point meets each
horizontal boundary in one circle and is a product annulus.  The critical
component also meets each horizontal boundary in one circle, is
nonorientable, and has Euler characteristic $-1$; it is therefore
$\RP^2\setminus(D^2\sqcup D^2)$.

Regluing the outgoing boundary circles to the incoming ones gives a
permutation of the $k$ components.  Since $N_3$ is connected, this
permutation is transitive, hence is a $k$-cycle.  The boundary-gluing
diffeomorphisms contain further orientation data, but this does not affect
the orbit decomposition of the set of components.  Collapsing each connected
level component now gives a circle with one marked valence-two vertex.  In
the local Reeb model supplied by the index-one Morse normal form, the
critical component remains connected on both sides.  The quotient is
therefore locally an interval at the critical level, and the height map is
locally one-to-one at its
valence-two vertex.  The map induced by $f$ on this
Reeb circle therefore has $k$ preimages over every regular point and is a
topological covering of degree $k$.

Write $p:N_3\to\mathcal R_f$ for the Reeb quotient and
$h:\mathcal R_f\to S^1$ for the induced covering.  Thus $f=h\circ p$.
A transverse arc from the incoming to the outgoing boundary can be chosen in
each product annulus.  In the critical component choose such an arc avoiding
the stable and unstable separatrices.  The cyclic boundary gluing joins these
arcs to a loop which maps once around $\mathcal R_f$.  Hence $p_*$ is
surjective on fundamental groups and on first homology, so $p^*\eta$ is
primitive.  Since $h^*\eta=k\eta$, the
last assertion follows.
\end{proof}

The class $f^*\eta$ is nonzero.  Indeed, if it vanished, then $f$ would lift
to a real-valued function on the closed surface $N_3$, which would have a
minimum and a maximum, contrary to the one-critical-point hypothesis.  Its
divisibility is therefore understood in the free abelian group
$H^1(N_3;\Z)\cong\Z^2$.

Every $k\geq1$ occurs for some such Morse map.  Start with the one-saddle
Morse cobordism $\RP^2\setminus(D^2\sqcup D^2)$ from one circle to one circle
and glue its boundary circles.  This gives $N_3$ and a primitive
circle-valued map.  Composing with the degree-$k$ covering of the circle gives
the claimed examples.  Thus the topology of $N_3$ alone does not force
$k=1$.

\begin{corollary}\label{cor:real-monodromy}
For the real Albanese map, let $k$ be the number of ovals of a smooth real
fiber.  Then the boundary-gluing permutation is a $k$-cycle, the real saddle
is a nonorientable self-attachment, and
\[
 k=\divis\left(\alpha_{\R}^*\eta\right),
 \qquad 1\leq k\leq20.
\]
In particular, the oval number can be determined by computing the
divisibility of the restricted Albanese class.
\end{corollary}

\begin{proof}
Stipsicz and Szab\'o prove that $X(\mathbb R)\cong N_3$, that
$\alpha_{\mathbb R}:X(\mathbb R)\to E(\mathbb R)$ is onto, and that it has
exactly one critical point, which is a nondegenerate saddle \cite{SS}.
Thus it is a circle-valued Morse map with one critical point of index one,
and every hypothesis of Theorem~\ref{thm:reeb} holds.  Apply that theorem
and use Harnack's inequality for a real curve of genus $19$.
\end{proof}

There is an equivalent group-theoretic formulation.  Write
$X=\Pi\backslash\mathbb B^2$, let $s$ be the reflection in
\eqref{eq:reflection-matrix}, and put $\phi(g)=sgs^{-1}$.  In the homology
coordinates
\[
 [a_1]=(1,3),\qquad [a_2]=(-2,1),\qquad [a_3]=(-1,-1),
\]
the invariant lattice is $H_1(X;\Z)^c=\Z(1,-1)$.

\begin{proposition}\label{prop:fixed-subgroup}
Let $\operatorname{ab}:\Pi\to H_1(X;\Z)$ be abelianization.  Then
\[
 k=[\Z(1,-1):\operatorname{ab}(\Fix\phi)].
\]
Equivalently, if $\ell(x,y)=-x+y$, then
\[
 k=\frac12\gcd\{\lvert\ell(\operatorname{ab}g)\rvert:
                         g\in\Fix\phi\}.
\]
Thus the oval problem is reduced to a fixed-subgroup, or centralizer,
calculation in the explicit Cartwright-Steger presentation.
\end{proposition}

\begin{proof}
Put $P=\Fix(s)\subset\mathbb B^2$ and choose the lift of the real base point
in $P$.  The stabilizer of $P$ in $\Pi$ is
\[
 C_\Pi(s)=\{g\in\Pi:sgs^{-1}=g\}=\Fix\phi.
\]
Indeed, an element commuting with $s$ preserves $P$.  Conversely, if $gP=P$,
then $gsg^{-1}$ and $s$ are reflections fixing $P$ pointwise, so their
product is a holomorphic isometry fixing the maximal totally real plane $P$
pointwise.  At a point of $P$, its derivative is the identity on $T_pP$ and,
by complex linearity, on $T_p\mathbb B^2=T_pP\otimes_{\R}\mathbb C$.
A ball isometry is determined by its value and derivative at one point, so
the product is the identity.

The reflection $s$ normalizes $\Pi$ and induces the fixed real structure on
$X$; it is the reflection selected above and the one used in the ancillary
exact matrix calculation.
The induced map $C_\Pi(s)\backslash P\to X(\R)$ is injective.  If $x,y\in P$
and $y=gx$, then $g^{-1}sgs$ lies in $\Pi$ and fixes $x$.  Point stabilizers
are finite, while $\Pi$ is torsion-free, so $g^{-1}sgs=1$.  Its image is the
component through the chosen base point.  To see this, lift a path in
$X(\R)$ beginning there to a path $z(t)$ in $\mathbb B^2$.  Since the action
of $\Pi$ is free and properly discontinuous, the relation
$sz(t)=\gamma(t)z(t)$ determines a unique, locally constant
$\gamma(t)\in\Pi$.  Since
$\gamma(0)=1$, the lifted path remains in $P$.  The real locus is connected
because $X(\mathbb R)\cong\#_3\mathbb{RP}^2$ by
Proposition~\ref{prop:input}; hence
\[
                         X(\R)=C_\Pi(s)\backslash P.
\]
Since $P$ is a real hyperbolic plane, it is simply connected, and therefore
$\pi_1(X(\R))=C_\Pi(s)=\Fix\phi$.

By \cite[Theorem~1.6]{CKY}, $H_1(X;\Z)=\Pi/[\Pi,\Pi]\cong\Z^2$, with the
displayed coordinates on $a_1,a_2,a_3$.  The period homomorphism is normalized
in \cite[Section~5.4]{CKY} as $(m,n)\mapsto m-n\omega$, an isomorphism from
this integral lattice onto $H_1(E;\Z)=\Z+\omega\Z$.  Thus the Albanese map
induces an integral isomorphism
\[
                         \alpha_*:H_1(X;\Z)\xrightarrow{\cong}H_1(E;\Z),
\]
not merely a rational one, and equivariance carries the invariant lattice
onto $H_1(E;\Z)^c$.

We compute the integral involution directly.  Exact conjugation by
the reflection matrix in \eqref{eq:reflection-matrix}, followed by rewriting
in the CKY generators, gives the following identities, also recorded in
\cite[Proposition~2.3]{SS}:
\begin{align*}
 \phi(a_1)&=a_1^{-1}a_2^{-1}a_1^{-1}a_3^{-3}a_2^3,\\
 \phi(a_2)&=a_1^{-1}a_2^{-1}a_1a_2^2a_1^{-1}a_2^{-1}
              a_1a_3^{-1}a_1^{-1}a_2a_1,\\
 \phi(a_3)&=a_3^{-1}.
\end{align*}
These identities are verified in the exact cyclotomic model by the ancillary
script.  Abelianizing and using the displayed CKY coordinates gives
\[
 [\phi(a_1)]=(-3,-1),\qquad
 [\phi(a_2)]=(-1,2),\qquad
 [\phi(a_3)]=(1,1).
\]
Thus in the CKY coordinates the involution has matrix
\[
 M=\begin{pmatrix}0&-1\\-1&0\end{pmatrix},
 \qquad \Fix(M)=\Z(1,-1).
\]
The quotient $E/\langle c_E\rangle$ is a M\"obius band, so $E(\mathbb R)$
is one circle.  Let $d\in H_1(E;\Z)$ be its class.  The fixed
circle lifts to the fixed line of the linear involution,
whose lattice of deck translations is
$\Fix(M)=\Z(1,-1)$.  The quotient of this line by its translation lattice
is exactly $E(\mathbb R)$, so inclusion of the real circle realizes the
whole invariant lattice.  Hence $d$ is a primitive generator and
\[
                       H_1(E;\Z)^{c_E}=\Z d.
\]
Thus no factor of two enters when the restricted real map is compared with
the invariant lattice.
Corollary~\ref{cor:real-monodromy} now gives the index formula.  Finally,
$\ell(1,-1)=-2$, which gives the displayed gcd formula.
\end{proof}

The first calculation produces an element whose image has index two in the
invariant lattice.  The proposition is computer-assisted.  We
use the generators $u,v,b$ of the ambient
arithmetic group in \cite{CKY}, and write capital letters for inverses.

\begin{proposition}\label{prop:k-one-two}
The element
\[
 W=u^{-1}v^{-1}uv^2u^{-1}vbu^{-1}v^{-1}uv^{-1}b^{-1}uv^{-1}
\]
belongs to $\Pi$, satisfies $\phi(W)=W$, and has
\[
                         [W]=(2,-2)\in H_1(X;\Z).
\]
Consequently $k$ divides two.
\end{proposition}

\begin{proof}
The ambient presentation on $u,v,b$ is \cite[Theorem~1.1]{CKY}.
The explicit words $a_1,a_2,a_3$, the equality
$\Pi=\langle a_1,a_2,a_3\rangle$, the index $864$, and the integral
abelianization coordinates are given in \cite[Theorem~1.6]{CKY}.  Use the
corresponding right-coset table for $\Pi$ in $\overline\Gamma$.
Successively reading the letters of
\[
                         W=UVuvvUvbUVuVBuV
\]
returns the trivial coset, and hence $W\in\Pi$.  In the exact matrix model
over $\mathbb Q(\zeta_{12})$, where $\zeta_{12}=e^{\pi i/6}$, use the
reflection matrix $K$ from \eqref{eq:reflection-matrix}.
Direct multiplication gives
\[
                         K\overline W K^{-1}=W
\]
in the projective unitary group.  The arithmetic realization in \cite{CKY}
is faithful, and this is precisely the reflection used to define $\phi$.
Thus $W\in\Fix\phi$.

For the homology calculation, use the $864$-sheeted presentation complex
associated with the same coset table.  There are two integer cellular
one-cocycles whose values on $(a_1,a_2,a_3)$ are respectively
\[
                         (1,-2,-1),\qquad(3,1,-1).
\]
Their sums around every lift of each defining relator of
$\overline\Gamma$ vanish, so they represent the two coordinate homomorphisms
$H_1(\Pi;\Z)\to\Z$ in the integral basis of
\cite[Theorem~1.6]{CKY}.  Their values on $W$ are
$2$ and $-2$.  Hence $[W]=(2,-2)$.  Proposition~\ref{prop:fixed-subgroup}
now shows that $k$ divides two.  Since $k$ is positive, the result follows.
The search for the integer cocycles uses numerical optimization.  The
supplied certificate is then verified independently using integer arithmetic
only, while the group-invariance calculation uses exact cyclotomic
arithmetic.  The coset table, the cocycles, and standalone verification
scripts accompany the paper.
\end{proof}

The second calculation excludes a primitive element in the image of the
fixed subgroup.  Let
$\mathbb F_4=\mathbb F_2(\omega)$, where $\omega^2+\omega+1=0$.  The
reduction used in \cite[Section~1.5]{CKY} sends $\zeta_{12}$ to $\omega$ and
gives
\[
 \rho_2:\overline\Gamma\longrightarrow PU(3,\mathbb F_4).
\]

\begin{theorem}\label{thm:finite-obstruction}
Every element of $\Fix\phi$ has trivial image in $H_1(X;\mathbb F_2)$.
Consequently
\[
 \operatorname{ab}(\Fix\phi)=2\mathbb Z(1,-1)
 \quad\text{and}\quad k=2.
\]
\end{theorem}

\begin{proof}
Combine the restriction of $\rho_2$ with mod-two abelianization:
\[
 q:\Pi\longrightarrow
 PU(3,\mathbb F_4)\times H_1(X;\mathbb F_2),
 \qquad q(g)=(\rho_2(g),[g]_2),
\]
and let $Q=q(\Pi)$.  Complex conjugation on $\mathbb Q(\zeta_{12})$
reduces to the Frobenius automorphism $x\mapsto x^2$ of $\mathbb F_4$.
The reduction of the matrix $K$ above therefore induces on $Q$ the
involution
\[
 (A,h)\longmapsto
 \left(KA^{(2)}K^{-1},Mh\right),
 \qquad
 M=\begin{pmatrix}0&1\\1&0\end{pmatrix}.
\]
Here $A^{(2)}$ is obtained by squaring every entry of $A$.  The matrix $M$
is the reduction modulo two of the action on $H_1(X;\mathbb Z)$.

The calculation is finite and exact.  Enter the reductions of the CKY
matrices for $a_1,a_2,a_3$, together with the homology labels
\[
 [a_1]_2=(1,1),\qquad [a_2]_2=(0,1),\qquad [a_3]_2=(1,1).
\]
Matrices are normalized by their first nonzero entry, which realizes the
projective quotient inside $PGL(3,\mathbb F_4)$; the enumerated CKY image is
the cited projective unitary subgroup.  The script also verifies, on all
three generators, the reduced word formulas for $\phi$.
Enumeration of the subgroup generated by these three labeled matrices gives
\[
 |Q|=72.
\]
The enumeration verifies that this rule is well defined on $Q$, preserves
$Q$, and squares to the identity.  It has six fixed elements.  Direct
enumeration of these six elements gives
\[
 \{h\in H_1(X;\mathbb F_2):(A,h)\in\Fix(Q)\}=\{(0,0)\}.
\]
If $g\in\Fix\phi$, then $q(g)$ is fixed by the induced involution, and hence
$[g]_2=0$.  Since $\operatorname{ab}(\Fix\phi)$ is contained in
$\mathbb Z(1,-1)$ by Proposition~\ref{prop:fixed-subgroup}, this proves
\[
 \operatorname{ab}(\Fix\phi)\subseteq2\mathbb Z(1,-1).
\]
Proposition~\ref{prop:k-one-two} supplies an element with class
$2(1,-1)$, so equality holds.  The index formula in
Proposition~\ref{prop:fixed-subgroup} now gives $k=2$.
\end{proof}

Let $B\subset Y$ be the branch surface and $j:B\hookrightarrow Y$ its
inclusion.  The quotient map identifies $X(\R)$ with $B$.

\begin{theorem}\label{thm:branch-image}
Under the isomorphism
\[
\overline\alpha_*:\pi_1(Y)\longrightarrow\pi_1(\mathbb M)\cong\Z
\]
recalled in Section~2, one has
\[
                         j_*\pi_1(B)=2k\Z.
\]
In particular, Theorem~\ref{thm:finite-obstruction} gives
\[
                         j_*\pi_1(B)=4\Z.
\]
\end{theorem}

\begin{proof}
The restriction of $\overline\alpha$ to $B$ takes values in
$\partial\mathbb M$ and agrees with $\alpha_{\R}$ under the identification
$X(\R)\cong B$.  Corollary~\ref{cor:real-monodromy} says that its image on
fundamental groups is $k\Z\subseteq\pi_1(\partial\mathbb M)$.  The boundary
circle of a M\"obius band represents twice its core, so the inclusion
$\pi_1(\partial\mathbb M)\to\pi_1(\mathbb M)$ is multiplication by two.
The commutative quotient diagram and the isomorphism $\overline\alpha_*$
now give the result.
\end{proof}

\section{Smooth real fibers}

Let $b\in E(\R)$ be the real critical value and let
$t\in E(\R)\setminus\{b\}$.  The fiber $F_t=\alpha^{-1}(t)$ is a smooth real
curve of genus $19$.  Its real locus is nonempty and has the form
\[
                 F_t(\R)=C_1\sqcup\cdots\sqcup C_k,
                 \qquad 1\leq k\leq20.
\]
The upper bound is Harnack's inequality.  Since
$E(\R)\setminus\{b\}$ is connected and the restriction is a proper
submersion there, the integer $k$ is independent of $t$.

Put $Q_t=F_t/\langle c\rangle$.  It has $k$ boundary components and
\[
 \chi(Q_t)=\frac12\chi(F_t)=1-19=-18.
\]
It is orientable precisely when the real curve $F_t$ is dividing.  If it is
orientable of genus $h$, then
\[
 2-2h-k=1-g(F_t),
\]
and hence a dividing real curve of genus $g$ has
$k\equiv g+1\pmod2$.

We identify the relevant point of the Albanese torus without using an
equation for the Albanese map.  Let $\Lambda=\Z+\omega\Z$, where
$\omega^2+\omega+1=0$, and identify $E$ with $\mathbb C/\Lambda$ using the
normalization of \cite[Section~5.4]{CKY}.
This differs by a change of coordinate on $E$ from a normalization in which
the real structure is written as ordinary complex conjugation.

\begin{lemma}\label{lem:common-fixed-point}
The real structure on $E$ is induced by
\[
                         z\longmapsto\omega\overline z.
\]
The point $O=0$ is the unique point of $E$ fixed by both the real structure
and the order-three automorphism induced by $\sigma$.
\end{lemma}

\begin{proof}
The point $O_1=\Pi(0)$ of \cite[Proposition~5.5(b)]{CKY} lies over
$p_0=0$.  It is real because the reflection in
\eqref{eq:reflection-matrix} fixes the ball origin.  Equivariance therefore
shows that the real structure on $E$ fixes $p_0$.

In the integral coordinates used above, the period map is
$(m,n)\mapsto m-n\omega$, and the real involution is
$(m,n)\mapsto(-n,-m)$.  The anti-linear map
$z\mapsto\omega\overline z$ has precisely this action on the two periods:
it sends $1$ to $\omega$ and $-\omega$ to $-1$.  Since both maps fix the
origin, this identifies the induced real structure.

By the proof of \cite[Lemma~5.4]{CKY}, the action of $\sigma$ on $E$ is multiplication
by $\omega^t$, with $t=1$ or $2$, and its fixed points are
\[
 p_j=\frac{j(2+\omega)}3+\Lambda,\qquad j=0,\pm1.
\]
The real involution interchanges $p_1$ and $p_{-1}$.  Indeed,
\[
 \omega\overline{p_1}
   =\frac{1+2\omega}{3}+\Lambda
   =-\frac{2+\omega}{3}+\Lambda=p_{-1},
\]
because $1+\omega\in\Lambda$.  Thus only $p_0=0$ is fixed by both
actions.
\end{proof}

Let $\sigma$ denote the holomorphic automorphism of order three and put
\[
                         F_1=\alpha^{-1}(O).
\]
The fiber $F_1$ is smooth by \cite{KY,Rito}.  It is
preserved by both $\sigma$ and $c$.  Let
\[
                         G_1=F_1/\langle\sigma\rangle
\]
be the abstract smooth quotient curve.  The relation
$c\sigma c=\sigma^{-1}$ shows that conjugation descends to a real structure
on $G_1$.

\begin{theorem}\label{thm:nondividing}
The cyclic quotient map restricts to a homeomorphism
\[
                         F_1(\R)\longrightarrow G_1(\R).
\]
The curve $G_1$ has genus six.  Consequently
\[
                         1\leq k\leq7,
\]
and the real Albanese fiber $F_1$ is nondividing.
\end{theorem}

\begin{proof}
In the Borisov-Yeung bicanonical model, write $[U_i]$ for the coordinate
point whose only nonzero coordinate is $U_i$.  Direct substitution in the
defining equations of \cite{BY} shows that $[U_7],[U_8],[U_9]$ lie on $X$;
this check is also included in the ancillary archive.  The diagonal action
of $\sigma$ fixes them.  They are real for the coordinatewise conjugation
recalled in the Introduction.  Since each point is fixed by $\sigma$, it is
also fixed by the selected reflection $c$ when coordinate conjugation is
$\sigma^j c$.  Thus the possible composition by the order-three action does
not change these real points.  Their images under $\alpha$ are
therefore fixed by both induced actions on $E$, so
Lemma~\ref{lem:common-fixed-point} places all three points on $F_1$.  By
\cite[Proposition~5.5(d)]{CKY}, the restriction of $\sigma$ to
$F_1$ has exactly three fixed points, so these exhaust its ramification.
Since the group has prime order, there are no other ramification points of
the degree-three quotient.  In particular, all three branch values on
$G_1$ are real.  Riemann-Hurwitz gives
\[
 2\cdot19-2=3(2g(G_1)-2)+3(3-1),
\]
and hence $g(G_1)=6$.

Let $y\in G_1(\R)$ be away from the branch locus.  Conjugation acts as an
involution on the three-point $\sigma$-orbit over $y$, so it fixes either one
or three points.  It cannot fix all three.  Indeed, if all three points in the
orbit of a nonfixed point $x$ were real, then $c(x)=x$ and
$c(\sigma x)=\sigma x$.  The relation $c\sigma=\sigma^{-1}c$ would give
$\sigma x=\sigma^{-1}x$, and hence $\sigma x=x$, a contradiction.  Thus there is exactly one
real point over $y$.  At a branch value the unique point above it is real.
At a real branch point there are compatible local coordinates in which the
map is $z\mapsto z^3$ and conjugation is $z\mapsto\bar z$; on the real loci
this is $t\mapsto t^3$.  The map on real loci is therefore a continuous
bijection between compact Hausdorff one-manifolds, and hence a homeomorphism.
Harnack's inequality on the genus-six curve gives $k\leq7$.

Suppose that $F_1$ were dividing.  A union of disjoint circles in an oriented
surface separates it exactly when its mod-two homology class vanishes.  For
the real locus of a connected real curve, conjugation interchanges the two
complementary sides, which is the classical dividing criterion.  Thus
$[F_1(\R)]=0$ in $H_1(F_1;\F_2)$.  The continuous proper map
$F_1\to G_1$ has a homological pushforward with $\F_2$-coefficients.  On each
real oval its restriction has mod-two degree one: away from ramification it
is a homeomorphism, and at a real ramification point its local form is
$t\mapsto t^3$, again of local degree one modulo two.  Consequently
\[
  0=q_*[F_1(\R)]=[G_1(\R)]\in H_1(G_1;\F_2).
\]
The same separation criterion shows that $G_1$ is dividing.  For a dividing real curve of genus $g$ with $k$ ovals, its
orientable quotient has Euler characteristic $1-g$; hence
$2h+k=g+1$ and $k\equiv g+1\pmod2$.  Applied to $F_1$ this makes $k$ even,
whereas applied to $G_1$ it makes the same integer $k$ odd, a contradiction.
\end{proof}

Let $s=\sigma|_{F_1}$ and $r=c|_{F_1}$.  The relations
$s^3=r^2=1$ and $rsr=s^{-1}$ define an action of
$S_3=\langle s,r\rangle$ on $F_1$.

\begin{theorem}\label{thm:S3-homology}
As a rational $S_3$-module,
\[
 H_1(F_1;\mathbb Q)
 \cong \mathbf 1^{\oplus6}\oplus
       \operatorname{sgn}^{\oplus6}\oplus V^{\oplus13},
\]
where $V$ is the two-dimensional standard representation.  In particular,
\[
             \dim H_1(F_1;\mathbb Q)^{\langle s\rangle}=12.
\]
\end{theorem}

\begin{proof}
The automorphism $s$ has exactly three fixed points on $F_1$ by
\cite[Proposition~5.5(d)]{CKY}.  Each has local Lefschetz index one, since
$s$ is holomorphic and acts nontrivially on the tangent line.  Hence
\[
 3=L(s)=1-\operatorname{tr}\bigl(s_*|H_1(F_1;\mathbb Q)\bigr)+1,
\]
so the trace of $s_*$ is $-1$.

Each of the three reflections in $S_3$ is antiholomorphic.  Its fixed set is
empty or a disjoint union of circles, and therefore has Euler characteristic
zero.  Since an antiholomorphic involution of a compact Riemann surface has
degree $-1$, the Lefschetz formula shows that its trace on
$H_1(F_1;\mathbb Q)$ is zero.

Let $a,b,d$ be the multiplicities of the trivial, sign, and standard
representations.  The character values at the identity, a three-cycle, and
a reflection give
\[
 a+b+2d=38,\qquad a+b-d=-1,\qquad a-b=0.
\]
Thus $a=b=6$ and $d=13$.  A three-cycle has no fixed vector in $V$ and acts
trivially on both one-dimensional representations, which gives the last
assertion.
\end{proof}

The theorem determines the rational representation but not the integral
$S_3$-lattice or the isotopy classes and intersections of the vanishing
cycles.  The resulting constraint applies to monodromy and transport maps
that commute with the indicated $S_3$-action; the order-three symmetry acts nontrivially
on the Albanese base.

Moreover, $G_1$ is nondividing.  Indeed, a dividing real curve of genus six
has an odd number of ovals, whereas $G_1(\mathbb R)$ has two components.
Since the smooth real fibers form a proper equivariant
submersion over $E(\R)\setminus\{b\}$.  This is one connected interval
because $E/\langle c_E\rangle$ is a M\"obius band.  Average a
horizontal distribution under $c$ to obtain an equivariant Ehresmann
trivialization.  Thus the pairs $(F_t,F_t(\R))$ are diffeomorphic, and every
smooth real Albanese fiber is nondividing and has the same value $k=2$.
Consequently $Q_t$ is nonorientable with two boundary components.  The
classification of compact surfaces gives
\[
 Q_t\cong N_{18}\setminus(D^2\sqcup D^2),
\]
where $N_{18}$ is the closed nonorientable surface of genus $18$.  At the
real critical value the fiber has a transverse real node \cite{KY,Rito}.
The fiber is irreducible and reduced \cite[Section~1.6]{Keum}, so its
normalization has genus $18$.  The local saddle must be combined with the
global boundary gluing around the remainder of $E(\R)$; the local crossing alone
does not determine the oval configuration.

\section{The Borisov-Yeung weight-two pencil}\label{sec:weight-two}

Let $F_1$ be the invariant smooth Albanese fiber used above and let
$H=2K_X$ be the bicanonical class.  The results of \cite{CKY,BY} give
\[
 H\cdot F_1=72,
\]
and the order-three action has weights
\[
 (0,0,0,0,1,1,1,2,2,2)
\]
on $(U_0,\ldots,U_9)$.  We use the Borisov-Yeung equations for $X$ and
the characteristic-zero equations published as ancillary files to
\cite{Rito}.  In Rito's normalization, the latter equations define the
invariant fiber through the three order-three fixed points
$[U_7],[U_8],[U_9]$; by Lemma~\ref{lem:common-fixed-point} and
\cite[Proposition~5.5(d)]{CKY}, this is the fiber $F_1=\alpha^{-1}(O)$ used
above.  The coordinate order and the diagonal weights are those of
\cite{BY}.  The ancillary archive to the present paper includes
byte-for-byte copies of the two source files, their upstream arXiv
identifiers and versions, and their checksums.

\begin{theorem}\label{thm:weight-two-pencil}
The weight-two bicanonical net
\[
                 \langle U_7,U_8,U_9\rangle|_{F_1}
\]
is base-point-free.  Consequently, every real pencil generated by a pair
with no common zero defines a $\sigma$-invariant morphism
\[
                 \ell_1:F_1\longrightarrow\mathbb P^1
\]
of degree $72$.  It factors through the cyclic quotient as
\[
 F_1\xrightarrow{q}G_1\xrightarrow{\bar\ell_1}\mathbb P^1,
 \qquad \deg q=3,\qquad \deg\bar\ell_1=24.
\]
The ramification divisor of $\bar\ell_1$ has degree $58$, counted with multiplicity.
The particular pair
\[
 L_0=U_7+U_8+U_9,
 \qquad
 L_1=U_7+2U_8+3U_9
\]
has no common zero on $F_1$ and therefore gives an explicit degree-$24$
quotient map.
\end{theorem}

\begin{proof}
Base-point-freeness is equivalent to
\[
 F_1\cap\{U_7=U_8=U_9=0\}=\varnothing.
\]
This is checked from the exact Borisov-Yeung equations and the equations of
the invariant fiber.  After clearing the denominators appearing in the
published equations, let $\mathcal Z\subset\mathbb P^9_{\mathbb Z}$ be the
closed subscheme defined by these homogeneous equations together with
$U_7=U_8=U_9=0$.  Its generic fiber is the base scheme in question.  Set
$U_7=U_8=U_9=0$ and reduce the resulting
homogeneous ideal modulo $10007$.  On each of the seven standard affine
charts $U_i=1$, $0\leq i\leq6$, an exact Gr\"obner-basis computation gives
the unit ideal.  These charts cover the locus $U_7=U_8=U_9=0$, so no
saturation assumption is needed and the projective special fiber is empty.
If the generic fiber of the projective scheme $\mathcal Z$ were nonempty,
properness would extend a closed point, after a finite field extension, over
a valuation ring above $10007$; its closed point would give a point of the
special fiber.  This contradiction proves that the characteristic-zero base
scheme is empty.  The same calculation modulo $10009$ and $10037$ gives
independent checks.  No smoothness or good-reduction hypothesis is needed
for this implication from an empty special fiber.

For the displayed pair $L_0,L_1$, their common zero in the weight-two
coordinate plane is $[U_7:U_8:U_9]=[1:-2:1]$.  Substitute
$(U_7,U_8,U_9)=(1,-2,1)$ into the same equations.  The resulting affine
ideal in $U_0,\ldots,U_6$ is the unit ideal modulo each of
$10007,10009,10037$.  The first projective calculation excludes the boundary
$U_7=U_8=U_9=0$.  On its complement, the equations $L_0=L_1=0$ force the
weight-two coordinates to be proportional to $(1,-2,1)$, so the displayed
normalization covers every remaining projective point.  The preceding
specialization argument therefore proves that the characteristic-zero
common-zero scheme is empty.

The bad pairs form a proper real algebraic subset of the space of pairs,
and the displayed real pair lies outside it.  Hence base-point-free real
pencils form a nonempty real open set.  For every such pair, the pullback of
$\mathcal O_{\mathbb P^1}(1)$ is $H|_{F_1}$, so its ratio has degree
\[
                         H\cdot F_1=72.
\]
Both sections have the same character, so their ratio is $\sigma$-invariant
and descends through $q$.  It follows that
\[
                         \deg\bar\ell_1=72/3=24.
\]
Since $g(G_1)=6$, Riemann-Hurwitz gives
\[
 \deg R_{\bar\ell_1}
   =2g(G_1)-2+2\deg\bar\ell_1=10+48=58.
\]
\end{proof}

Writing $p_7=[U_7]$, $p_8=[U_8]$, and $p_9=[U_9]$ for the three points on
$F_1$, the proof of Theorem~\ref{thm:nondividing} and
\cite[Proposition~5.5(d)]{CKY} show that these are exactly the ramification
points of $q$, each with index three.  The composition formula for
ramification divisors gives
\[
 R_{\ell_1}
   =2p_7+2p_8+2p_9+q^*R_{\bar\ell_1}.
\]
Thus, after removing the forced ramification at the three fixed points, the
residual ramification divisor on $F_1$ has degree
\[
                         180-6=174=3\cdot58
\]

Any further marked ramification calculation at $F_1$ can therefore be carried
out on $24$ sheets rather than on the degree-$72$ map upstairs.
The genus-$19$ fiber is recovered from the cyclic triple cover
$F_1\to G_1$, branched at its three fixed points.  This does not replace a
global calculation over the Albanese base, because $\sigma$ acts
nontrivially on $E$.

\section{Circle surgery}

For an embedded oriented loop $\gamma$ in an oriented four-manifold $M$,
let $\nu\gamma$ denote a closed tubular neighborhood of $\gamma$, and choose
an orientation-preserving normal framing
\[
 f:S^1\times D^3\longrightarrow\nu\gamma.
\]
The two homotopy classes of such framings form a torsor for
$\pi_1(SO(3))\cong\mathbb Z/2$.  Put
$C=M\setminus\operatorname{int}\nu\gamma$.  Using the framing on the common
boundary, circle surgery is
\begin{equation}\label{eq:surgery}
 M_{\gamma,f}=C\mathop{\cup}_{\varphi_f}(D^2\times S^2),
\end{equation}
where $\varphi_f:\partial(D^2\times S^2)\to\partial C$ is the resulting
orientation-reversing identification.  Thus the framing is part of the
gluing data.  We write $M_\gamma^\varepsilon$, $\varepsilon\in\mathbb Z/2$,
for the two choices.  The smooth diffeomorphism type may depend on
$\varepsilon$.

\begin{lemma}\label{lem:circle}
Let $M$ be a closed connected oriented four-manifold and let $\gamma$ be an
embedded loop.  Then
\[
 e(M_\gamma)=e(M)+2,\qquad \sigma(M_\gamma)=\sigma(M),
\]
and
\[
 \pi_1(M_\gamma)\cong
 \pi_1(M)/\langle\!\langle[\gamma]\rangle\!\rangle.
\]
\end{lemma}

\begin{proof}
The Euler-characteristic formula follows from additivity, since
$e(S^1\times D^3)=e(S^1\times S^2)=0$ and $e(D^2\times S^2)=2$.
The two boundary components of the five-dimensional surgery trace are
oriented cobordant and hence have equal signatures; equivalently, this follows
from Novikov additivity for the relative intersection pairings.  General
position shows that $M\setminus\operatorname{int}\nu\gamma\to M$ induces an
isomorphism on fundamental groups: loops can be moved off the
codimension-three submanifold, and transverse nullhomotopies miss it.  Under
this isomorphism the $S^1$ factor of the boundary
$S^1\times S^2$ represents
$[\gamma]$.  Van Kampen's theorem then gives the stated quotient.
\end{proof}

The surgery trace is obtained from $M\times I$ by attaching a
five-dimensional index-two handle $D^2\times D^3$ along
$S^1\times D^3$.  Its belt sphere in the outgoing boundary is
$\{0\}\times S^2$.  The boundary identification inherited from this handle
is precisely the datum required to reverse the surgery.

Let $Q$ be a closed oriented smooth four-manifold satisfying
\[
 H_*(Q;\mathbb Z)\cong H_*(S^1\times S^3;\mathbb Z),
 \qquad \pi_1(Q)\cong\mathbb Z.
\]

\begin{proposition}
\label{prop:abstractblowup}
Let $Q$ be as above, let $\gamma\subset Q$ represent a
generator of $\pi_1(Q)$, and choose either normal framing
$\varepsilon\in\mathbb Z/2$.  Form connected sums in a ball disjoint from
$\gamma$, and put
$\Sigma_Q^\varepsilon=Q_\gamma^\varepsilon$.  Then
\[
 \Sigma_Q^\varepsilon\simeq S^4,
 \qquad
 (Q\#\CP^2)_\gamma^\varepsilon
       \cong \Sigma_Q^\varepsilon\#\CP^2.
\]
In particular, $(Q\#\CP^2)_\gamma^\varepsilon$ is
orientation-preservingly homeomorphic to $\CP^2$.  If
$Q\cong S^1\times S^3$, then $\Sigma_Q^\varepsilon\cong S^4$ for both
framings, and the surgery output is diffeomorphic to $\CP^2$.
\end{proposition}

\begin{proof}
Lemma~\ref{lem:circle} gives
\[
 \pi_1(\Sigma_Q^\varepsilon)=1,
 \qquad e(\Sigma_Q^\varepsilon)=2,
 \qquad \sigma(\Sigma_Q^\varepsilon)=0.
\]
Since $\pi_1=1$, one has $H_1=0$, and Poincar\'e duality gives $H_3=0$.
The universal coefficient theorem and Poincar\'e duality show that $H_2$ is
torsion free.  The equality $e=2$ then gives $H_2=0$.  Thus
$\Sigma_Q^\varepsilon$ is an integral homology four-sphere.  The Hurewicz
and Whitehead theorems \cite[Theorems~4.32 and~4.5]{Hatcher} make it a homotopy
four-sphere.  Since the connected-sum ball is disjoint from $\nu\gamma$, the
surgery is performed entirely in the $Q$ summand.  Thus
$(Q\#\CP^2)_\gamma^\varepsilon\cong
Q_\gamma^\varepsilon\#\CP^2$, proving the displayed diffeomorphism.  The output is simply connected with intersection
form $\langle1\rangle$ and vanishing Kirby--Siebenmann invariant, so
Freedman's theorem gives the homeomorphism with $\CP^2$.  When
$Q\cong S^1\times S^3$, Proposition~\ref{prop:generatorisotopy} shows that
$\gamma$, after reversing its orientation if necessary, is ambiently isotopic
to the product generator.  The exterior of that generator is
$S^1\times D^3$.  The untwisted filling is
$\partial(D^2\times D^3)\cong S^4$.  The two attaching maps differ by
$\tau(\theta,x)=(\theta,A_\theta x)$ on $S^1\times S^2$, where
$A_\theta$ represents the nontrivial loop in $SO(3)$.  The same formula on
$S^1\times D^3$ extends $\tau$ over the exterior.  Thus both framed surgeries give
$S^4$.
\end{proof}

\begin{proposition}
\label{prop:generatorisotopy}
Let $M$ be a connected oriented smooth four-manifold.  Two oriented embedded
loops which represent the same conjugacy class in $\pi_1(M)$ are ambiently
isotopic.  Consequently the circle-surgery manifolds obtained from two
Albanese generators in $Y$ are diffeomorphic once the corresponding normal
framing classes are fixed.
\end{proposition}

\begin{proof}
Since the two oriented loops represent the same conjugacy class in
$\pi_1(M)$, they are freely homotopic.  Choose a smooth free homotopy
$H:S^1\times I\to M$ relative to $S^1\times\partial I$.  We perturb $H$, relative to
$S^1\times\partial I$, through level-preserving maps
$(x,t)\mapsto(H_t(x),t)$.

We use relative level-preserving multijet transversality for this family,
with the perturbation fixed for $t$ near $0$ and $1$; the one-jet and deleted
two-point conditions below may be imposed simultaneously
\cite[Chapter~II]{GG}.  The stratum of rank-zero
one-jets has codimension four, while the parameter domain $S^1\times I$ has
dimension two.  The perturbed family therefore has no singular points, so
every $H_t$ is an immersion.  For self-intersections,
consider the deleted configuration space
\[
 \mathcal C=\{(x,y,t)\in S^1\times S^1\times I:x\ne y\}
\]
and the map
\[
 F:\mathcal C\longrightarrow M\times M,\qquad
 F(x,y,t)=(H_t(x),H_t(y)).
\]
A double point of some $H_t$ is precisely a point of $F^{-1}(\Delta_M)$,
where $\Delta_M\subset M\times M$ is the diagonal.  The space $\mathcal C$
has dimension three, whereas $\Delta_M$ has codimension four in
$M\times M$.  The same relative parametric multijet perturbation may be chosen so that
$F$ is transverse to $\Delta_M$ away from the deleted diagonal.  Its expected
inverse-image dimension is $3-4=-1$, so this inverse image is empty.
Because the derivative in the $S^1$ direction is nowhere zero on the compact
set $S^1\times I$, there is a neighborhood of the diagonal in
$S^1\times S^1\times I$ on which $H_t(x)\ne H_t(y)$ for $x\ne y$.
Transversality on the complement eliminates all remaining double points.
Thus the perturbation, relative to the endpoints, simultaneously removes
singular points and double points.  Hence every $H_t$ is an embedding, and
$H_t$ is an isotopy of the two oriented loops.

By the isotopy extension theorem this isotopy extends to an ambient isotopy
of $M$.  Transporting a chosen normal framing along the ambient isotopy
identifies the corresponding framed tubular neighborhoods
$S^1\times D^3$.  Replacing these neighborhoods by $D^2\times S^2$ therefore
gives diffeomorphic circle-surgery manifolds.  This proves the final
assertion.
\end{proof}

\begin{proposition}\label{prop:one}
Let $\gamma\subset Y$ represent a generator of $\pi_1(Y)$.  For either normal
framing, $Y_\gamma$ is a smooth manifold homotopy equivalent to $\CP^2$ and is
orientation-preservingly homeomorphic to $\CP^2$.
\end{proposition}

\begin{proof}
By Lemma~\ref{lem:circle}, $Y_\gamma$ is simply connected and
\[
 e(Y_\gamma)=3,\qquad \sigma(Y_\gamma)=1.
\]
Thus $b_2=1$ and its unimodular intersection form is the positive odd form
$\langle1\rangle$.  Freedman's classification theorem \cite{FQ}, together with the vanishing
Kirby--Siebenmann invariant of the smooth manifold $Y_\gamma$, gives the
claimed homeomorphism.  The homotopy conclusion follows immediately.
\end{proof}

\begin{corollary}
\label{cor:standardquotient}
If $Y\cong\CP^2\#(S^1\times S^3)$, then circle surgery with either framing on an
Albanese generator is diffeomorphic to $\CP^2$.  Hence a nonstandard output
would imply that $Y$ itself is nonstandard.
\end{corollary}

\begin{proof}
Under a diffeomorphism with the standard connected sum, the generator loop is
ambiently isotopic, by Proposition~\ref{prop:generatorisotopy}, to the product
circle in the $S^1\times S^3$ summand; if its orientation is reversed, use the
product reflection on the two factors.  Its exterior within that summand is
$S^1\times D^3$, and
the untwisted filling is
\[
 (S^1\times D^3)\cup(D^2\times S^2)=\partial(D^2\times D^3)\cong S^4.
\]
The boundary twist distinguishing the other framing is the map $\tau$ above
and extends over $S^1\times D^3$ by the same loop of rotations of $D^3$.  The second filling
is therefore also $S^4$, and the full surgery is
$\CP^2\#S^4\cong\CP^2$.
\end{proof}

\subsection{The dual sphere}
In the five-dimensional surgery trace, the belt sphere of the index-two
handle is the embedded sphere
\[
                         S=\{0\}\times S^2\subset Y_\gamma
\]
with its canonical trivial normal bundle in the surgery piece.  Removing
$\nu S=S^2\times D^2$ recovers the circle exterior, and regluing
$S^1\times D^3$ by the boundary map inherited from the chosen circle framing
recovers $Y$:
\[
 Y=\bigl(Y_\gamma\setminus\operatorname{int}(S^2\times D^2)\bigr)
       \cup (S^1\times D^3).
\]
By a \emph{parametrized dual sphere} we mean $S$ together with the inherited
identification
\[
 \rho_f:S^2\times S^1\longrightarrow
 \partial(Y_{\gamma,f}\setminus\operatorname{int}\nu S),
\]
where two identifications are equivalent if they differ by diffeomorphisms
extending over $S^2\times D^2$ on the source side and over
$S^1\times D^3$ after reverse surgery on the target side.  Since
$[S^2,SO(2)]=0$, an oriented trivial rank-two bundle over $S^2$ has a unique
homotopy class of trivialization.  The additional datum is therefore the
reverse-surgery boundary parametrization, not another normal framing of $S$.
Moreover, the sphere complement is precisely the circle complement,
\[
 Y_\gamma\setminus\operatorname{int}\nu S
 \cong Y\setminus\operatorname{int}\nu\gamma.
\]
The parametrized pair $(Y_\gamma,S)$ retains more of the surgery data than the
closed manifold $Y_\gamma$ alone.

The sphere $S$ is nullhomologous.  Indeed, it has square zero, whereas the
intersection form of $Y_\gamma$ is $\langle1\rangle$; since
$H_2(Y_\gamma;\mathbb Z)$ is torsion free, this forces $[S]=0$.  Thus the
remaining information concerns a knotted nullhomologous sphere and its parametrized
complement, rather than a new homology class.

General position gives
$\pi_1(Y\setminus\nu\gamma)\cong\pi_1(Y)\cong\mathbb Z$, but it does not
determine the peripheral system or the boundary parametrization relevant to
the dual surgery.

\begin{proposition}\label{prop:gluck}
Let $P^0$ and $P^1$ be the results of circle surgery on $\gamma$ using the two
normal framings.  If $S\subset P^0$ is the dual sphere, then $P^1$ is obtained
from $P^0$ by the Gluck twist along $S$.
\end{proposition}

\begin{proof}
Choose product coordinates $\partial\nu\gamma=S^1\times S^2$.  The two
normal framings differ by
\[
 \tau(e^{i\theta},x)=(e^{i\theta},r_\theta(x)),
\]
where $\theta\mapsto r_\theta$ represents the nonzero element of
$\pi_1(SO(3))$.  After interchanging factors to identify this boundary with
$\partial(S^2\times D^2)=S^2\times S^1$, the same map is
\[
 (x,e^{i\theta})\longmapsto(r_\theta(x),e^{i\theta}),
\]
the Gluck regluing map.  Thus replacing one circle framing by the other is
exactly a Gluck twist on the dual sphere \cite{Gluck}.
\end{proof}

Thus Question~\ref{q:one} is a problem about one framed surgery together with
the Gluck twist on its dual sphere, rather than two unrelated closed
four-manifolds.

\subsection{Equivariant surgery upstairs}
Let $B\subset Y$ be as above.  If an Albanese generator
$\gamma$ is chosen disjoint from $B$, its inverse image is an unbranched
double cover of a circle.  It is therefore either one circle mapping with
degree two or two circles each mapping with degree one; the alternative is
determined by the covering monodromy of $\gamma$ in $Y\setminus B$.

For an Albanese generator this alternative can be decided.

\begin{theorem}
\label{thm:covermonodromy}
Every loop in $\overline\alpha^{-1}(\operatorname{int}\mathcal M)$ whose
image in the M\"obius band
represents an odd power of its core has nontrivial covering monodromy.  In
particular, the inverse image of an Albanese generator is one circle mapping
with degree two.
\end{theorem}

\begin{proof}
The Albanese map is equivariant with respect to $c$ on $X$ and complex
conjugation $c_E$ on the elliptic curve $E$.  Away from the fixed sets it
gives a pullback square of unbranched double covers
\[
\begin{CD}
 \alpha^{-1}(E\setminus E^{c_E}) @>{\alpha}>> E\setminus E^{c_E}\\
 @V{q_X}VV @VV{q_E}V\\
 \overline\alpha^{-1}(\operatorname{int}\mathcal M)
 @>{\overline\alpha}>> \operatorname{int}\mathcal M .
\end{CD}
\]
To verify that the square is Cartesian, fix $y\in
\overline\alpha^{-1}(\operatorname{int}\mathcal M)$ and a lift
$e\in E\setminus E^{c_E}$ of $\overline\alpha(y)$.  The fiber
$q_X^{-1}(y)$ is a free $c$-orbit, since the fixed set maps to
$\partial\mathcal M$ whereas $y$ lies over $\operatorname{int}\mathcal M$.
Equivariance shows that the Albanese
images of its two points are $e$ and $c_E(e)$, so exactly one point maps to
$e$.  This gives the required identification with the fiber product.
The cover on the right is the orientation double cover of the M\"obius band.
Its monodromy takes the core generator of
$\pi_1(\mathcal M)\cong\mathbb Z$ to the nonzero element of $\mathbb Z/2$.
If a loop in $\overline\alpha^{-1}(\operatorname{int}\mathcal M)$ had
trivial monodromy for $q_X$, it would have a
closed lift to $X\setminus X^c$; applying $\alpha$ would give a closed lift of
its image in $\operatorname{int}\mathcal M$.  This is impossible for an odd
power of the core.  Since
$\overline\alpha_*:\pi_1(Y)\to\pi_1(\mathcal M)$ is an isomorphism, the
assertion applies to the Albanese generator chosen above.
\end{proof}

Thus the lifted loop is connected.  The induced connected double cover of
$S^1\times S^2$ cannot extend as an
unbranched cover of the simply connected manifold $D^2\times S^2$; a lift
must therefore contain branching.  In the pulled-back product coordinates
there is a standard branched extension.

\begin{theorem}
\label{thm:branchedlift}
With the notation above, let
$\widetilde P=X_{\widetilde\gamma}$ be circle surgery using the pullback of
the chosen downstairs framing.  Then the involution extends over the surgery,
and its quotient is
$P=Y_\gamma$.  More precisely,
\[
                \widetilde P\longrightarrow P
\]
is a double branched cover with branch locus
\[
                         B\sqcup S,
\]
where $S\subset P$ is the dual sphere of the circle surgery.  Moreover,
\[
 e(\widetilde P)=5,\qquad \sigma(\widetilde P)=1,
 \qquad
 \pi_1(\widetilde P)=
 \pi_1(X)/\langle\!\langle[\widetilde\gamma]\rangle\!\rangle.
\]
The two framings downstairs induce the same framing class on
$\widetilde\gamma$.  For each downstairs framing the pullback construction
therefore gives a branched covering of the stated form; the underlying
upstairs surgery manifolds are diffeomorphic.  We make no identification of
the two extended involutions.
\end{theorem}

\begin{proof}
The original branch surface $B$ is disjoint from $\nu\gamma$ because
$\gamma\subset\overline\alpha^{-1}(\operatorname{int}\mathcal M)$ whereas
$B$ maps to $\partial\mathcal M$.  Thus the construction below leaves $B$
unchanged and creates one disjoint branch component in the filling.
Over $\nu\gamma=S^1\times D^3$, the connected double cover is again
$S^1\times D^3$, with the degree-two map on the first factor.  In framed
coordinates the boundary cover is
\[
 S^1\times S^2\longrightarrow S^1\times S^2,
 \qquad(e^{i\phi},x)\longmapsto(e^{2i\phi},x).
\]
It extends over the filling by
\[
 D^2\times S^2\longrightarrow D^2\times S^2,
                 \qquad (z,x)\longmapsto(z^2,x).
\]
This map is compatible with the attaching map determined by the pulled-back
framing and is branched precisely along $\{0\}\times S^2$.  Gluing it to
$q_X$ on the exterior gives the asserted branched cover and identifies the
new branch component with the dual sphere.  The boundary restriction is
exactly the pullback of the chosen downstairs attaching map, so the two
pieces glue compatibly.  Here $S\cdot S=0$, so
the normal bundle of $S=\{0\}\times S^2$ is the product bundle coming from
the oriented $D^2$ factor and has Euler number zero, compatibly with the
normal-Euler convention fixed in Section~2.  Moreover,
$\chi(B)=2-3=-1$ and $\chi(B\sqcup S)=-1+2=1$.  Hence
$e_N(B\sqcup S)=e_N(B)=2$; the branched-cover formulas give
$e(\widetilde P)=2e(P)-\chi(B\sqcup S)=5$ and
$\sigma(\widetilde P)=2\sigma(P)-e_N(B\sqcup S)/2=1$.
The defining signature formula for $X\to Y$ gives $e_N(B)=2$.
The same numerical formulas, and the fundamental-group formula, follow from
Lemma~\ref{lem:circle} applied upstairs.  Finally, changing the framing
downstairs adds the generator of $\pi_1(SO(3))\cong\mathbb Z/2$.  Pullback by
the degree-two map of the circle gives twice that generator, hence zero.  The
two pulled-back framings are homotopic, so the upstairs surgery manifolds are
diffeomorphic relative to their common exterior.  This relative
diffeomorphism need not conjugate the two extended involutions.
\end{proof}

\begin{proposition}
\label{prop:upstairshomology}
For an Albanese generator chosen as above, $[\widetilde\gamma]$ is primitive
in $H_1(X;\mathbb Z)\cong\mathbb Z^2$.  Consequently
\[
 H_0(\widetilde P)\cong H_4(\widetilde P)\cong\mathbb Z,
 \quad H_1(\widetilde P)\cong H_3(\widetilde P)\cong\mathbb Z,
 \quad H_2(\widetilde P)\cong\mathbb Z^5,
\]
and
\[
                   b_2^+(\widetilde P)=3,
                   \qquad b_2^-(\widetilde P)=2.
\]
\end{proposition}

\begin{proof}
Cartwright, Koziarz, and Yeung identify the map induced by the Albanese
fibration as an isomorphism
\[
       \alpha_*:H_1(X;\mathbb Z)\longrightarrow H_1(E;\mathbb Z)
       \cong\mathbb Z^2
\]
\cite[Section~5.4]{CKY}; see also \cite[Proposition~2.7]{SS}.  The nontrivial
deck translation of the connected double cover
$\widetilde\gamma\to\gamma$ is the half-rotation of a circle and therefore
preserves its orientation.  Hence
$c_*[\widetilde\gamma]=[\widetilde\gamma]$.  The invariant sublattice of
complex conjugation can be read directly from the M\"obius-band quotient.
Choose a basis $e_1,e_2$ of $H_1(E;\mathbb Z)$ adapted to its orientation
double cover, so that
\[
 c_{E*}e_1=e_1,\qquad c_{E*}e_2=e_1-e_2.
\]
It follows that $H_1(E;\mathbb Z)^{c_{E*}}=\mathbb Ze_1$, with $e_1$
primitive.  The class $e_1$ is represented by the lift of the core and also
by a component over the boundary.  Since the boundary of a M\"obius band
represents twice its core, after choosing the sign of $e_1$ one has
$q_{E*}e_1=2[\delta]$.
Let $\widetilde\delta\subset E\setminus E^{c_E}$ be the connected lift of
$\delta$ representing $e_1$.
Choose orientations so that $\overline\alpha|_\gamma$ has degree one.
Commutativity and connectedness give
$\alpha(\widetilde\gamma)\subset q_E^{-1}(\delta)=\widetilde\delta$.
Commutativity gives
\[
 \deg(q_E|_{\widetilde\delta})\,
 \deg(\alpha|_{\widetilde\gamma})
 =\deg(\overline\alpha|_\gamma)\,
  \deg(q_X|_{\widetilde\gamma}).
\]
Both vertical degrees are two, and therefore
$\deg(\alpha|_{\widetilde\gamma})=1$.  It follows that
$q_{E*}\alpha_*[\widetilde\gamma]=2[\delta]=q_{E*}e_1$.
On the invariant primitive generator, $q_{E*}$ is multiplication by two;
in particular it is injective on the rank-one invariant sublattice.  It
follows that $\alpha_*[\widetilde\gamma]=e_1$.  Since $\alpha_*$ is an
isomorphism, $[\widetilde\gamma]$ is primitive in $H_1(X)$.

Circle surgery quotients first homology by this class, so
$H_1(\widetilde P)\cong\mathbb Z$.  This homology group is torsion free, and Poincar\'e
duality and the universal coefficient theorem give
$H_3(\widetilde P)\cong\mathbb Z$ and torsion-free $H_2$.  Since
$e(\widetilde P)=5$ and $\sigma(\widetilde P)=1$, the remaining assertions
follow.
\end{proof}

\begin{proof}[Proof of Theorem~\ref{thm:circle}]
Proposition~\ref{prop:one} gives the homeomorphism
$Y_\gamma\approx\CP^2$.  Theorem~\ref{thm:covermonodromy} shows that
$\widetilde\gamma$ is connected.  Theorem~\ref{thm:branchedlift} gives
the branched covering, its branch locus, and the framing assertion.
Proposition~\ref{prop:upstairshomology} gives the stated homology and
intersection-form data.
\end{proof}

\section{The branch surface and the positive class}

We first record two standard facts about a finite branched quotient.

\begin{lemma}\label{lem:rationalquotient}
Pullback identifies rational cohomology downstairs with the invariant
cohomology upstairs:
\[
 q_X^*:H^2(Y;\mathbb Q)\longrightarrow H^2(X;\mathbb Q)^{c^*}.
\]
\end{lemma}

\begin{proof}
Choose a $c$-equivariant triangulation and pass to a barycentric subdivision,
so no simplex is inverted by its stabilizer.  The rational cellular chains
of the quotient identify with the coinvariants of the cellular chains of
$X$.  Averaging over $\langle c\rangle$ identifies coinvariants with
invariants in characteristic zero.  Since taking invariants is exact for a
finite group over $\mathbb Q$, dualizing gives
$H^*(Y;\mathbb Q)\cong H^*(X;\mathbb Q)^{c^*}$, with the isomorphism induced
by pullback; compare \cite[Chapter~III, Theorem~2.4]{Bredon}.
\end{proof}

\begin{lemma}\label{lem:branchedthom}
Let $F\subset Y$ be an oriented embedded surface transverse to $B$, and
orient $\widehat F=q_X^{-1}(F)$ by the local branched-cover model.  Then
\[
 \operatorname{PD}_X[\widehat F]
   =q_X^*\operatorname{PD}_Y[F],
 \qquad [\widehat F]^2=2F^2.
\]
\end{lemma}

\begin{proof}
Away from $B$ this is the ordinary pullback identity.  Near a point of
$F\cap B$, choose oriented smooth source coordinates $(z,w)$ and target
coordinates $(Z,W)$
in which
\[
 q_X(z,w)=(Z,W)=(z^2,w),\qquad B=\{Z=0\},\qquad F=\{W=0\}.
\]
Thus $\widehat F=\{w=0\}$ is smooth at the branch points; this is where
transversality of $F$ to $B$ is used.
The Thom class in the normal $w$-direction pulls back with multiplicity one,
so the identity extends across the branch points.  The formula for the square
follows by evaluating the cup product and using
$q_{X*}[X]=2[Y]$.
\end{proof}

\begin{proposition}
\label{prop:branchmodtwo}
The branch surface is nullhomologous modulo two:
\[
 [B]=0\in H_2(Y;\mathbb Z/2).
\]
Consequently every closed surface in $Y$ has even mod-two intersection with
$B$.
\end{proposition}

\begin{proof}
The normal Euler number $e_N(B)$ is the evaluation of the twisted Euler
class of the normal bundle of the nonorientable surface $B$ on its twisted
fundamental class.  Its reduction modulo two is the mod-two self-intersection
of $B$: this is the standard identity between the mod-two reduction of the
twisted Euler class of a real rank-two bundle and the top
Stiefel--Whitney class of that bundle.  The universal coefficient theorem and Proposition~\ref{prop:input} give
$H_2(Y;\mathbb Z/2)\cong\mathbb Z/2$.  The reduction of the integral
positive generator has square one, so the mod-two intersection form is
$\langle1\rangle$.  With the normal-Euler convention used here, Viro's
branched-cover signature formula \cite{Viro} is
\[
 \sigma(X)=2\sigma(Y)-\frac12e_N(B).
\]
Thus $e_N(B)=2$.  The mod-two self-intersection of $B$ is the reduction of
its normal Euler number, so $[B]^2=0$.  The only square-zero element of the
one-dimensional form $\langle1\rangle$ is zero.
\end{proof}

\begin{theorem}
\label{thm:spherebranch}
Suppose that the positive generator is represented by a smooth $+1$-sphere
$E\subset Y$.  After making $E$ transverse to $B$, put
$r=|E\cap B|$.  Then
\[
 r\geq6.
\]
\end{theorem}

\begin{proof}
Proposition~\ref{prop:branchmodtwo} shows that $r$ is even.  If $r=0$, the
restricted unbranched double cover of the simply connected sphere is
trivial.  Thus $q_X^{-1}(E)=E_1\sqcup E_2$, and each restriction
$q_X|_{E_i}:E_i\to E$ is a diffeomorphism.  The normal bundle of $E_i$ is the
pullback of the normal bundle of $E$, so $E_i^2=E^2=1$.  In particular each
$E_i$ is an essential sphere of nonnegative square in $X$.

Suppose that $r>0$.  A meridian of any branch point has nontrivial monodromy,
so the two sheets belong to one component.  The local model in
Lemma~\ref{lem:branchedthom} shows that
$\widehat E=q_X^{-1}(E)$ is a connected embedded oriented surface and that
\[
 g(\widehat E)=\frac r2-1,\qquad
 [\widehat E]^2=2.
\]
The Cartwright--Steger surface is a minimal K\"ahler surface of general type
with $b_2^+(X)=3$.  In the convention used here $K_X$ is a basic class with
nonzero degree-zero Seiberg--Witten invariant; changing conventions replaces
it by $-K_X$ and does not affect the argument \cite{TaubesSW}.  Fintushel and Stern
\cite[Lemma~5.1]{FSSphere} show that an essential embedded sphere of square
zero in a four-manifold with $b_2^+>1$ forces all Seiberg--Witten invariants
to vanish.  Their argument uses only a nonzero chamber-independent
degree-zero invariant and $b_2^+>1$, and therefore applies when $b_1>0$, in
particular to $b_1(X)=2$.  The same obstruction applies to a sphere of positive square
$n$: blow up $n$ distinct points on the sphere.  Its proper transform is an
essential square-zero sphere, while $b_2^+$ is unchanged.  By the
Seiberg--Witten blow-up formula \cite[Theorem~14.1.1]{GS}, the classes
$K_X\pm E_1\pm\cdots\pm E_n$ are basic and retain the nonzero degree-zero
invariant.  Thus a nonzero basic class of the original manifold gives a
nonzero basic class of the blow-up,
contradicting \cite[Lemma~5.1]{FSSphere}.  Here $b_2^+(X)=3>1$, and each
spherical class under consideration is non-torsion because its square is
$1$ when $r=0$ and $2$ when $r=2$.  Thus the nonzero canonical invariant
gives a contradiction, ruling out $r=0$ and $r=2$.

For the canonical basic class the expected dimension is zero, since
\[
 d(K_X)=\frac{K_X^2-(2e(X)+3\sigma(X))}{4}
       =\frac{9-(6+3)}4=0.
\]
Since $b_1(X)=2$, we use the positive-$b_1$ adjunction inequality
\cite[Theorem~1.3]{OSAdj}, rather than Theorem~1.1 of that paper.  Take
$\ell=0$ and the degree-zero algebra insertions $a=b=1$.  The canonical class has
expected dimension zero and $SW_{X,K_X}(1)\ne0$, so its degree-zero case
applies to a connected embedded oriented surface of positive genus
representing a non-torsion class of nonnegative square.

For $r\geq4$, the adjunction inequality
\cite[Theorem~1.3]{OSAdj} gives
\[
 2g(\widehat E)-2\geq [\widehat E]^2+
             |K_X\cdot[\widehat E]|.
\]
The class $[\widehat E]$ is invariant under $c_*$.  Since $c$ is
antiholomorphic, its differential identifies $c^*T^{1,0}X$ with
$T^{0,1}X$.  Hence
$c^*c_1(K_X)=-c_1(K_X)$.  Thus the notation $c^*K_X=-K_X$ means equality
of cohomology classes, not an isomorphism of holomorphic line bundles.
Since $c$ preserves the real orientation of \(X\), it
preserves the intersection pairing, and hence
$K_X\cdot[\widehat E]=0$.  Substitution gives
\[
 r-4\geq2,
\]
so $r\geq6$.
\end{proof}

Put
\[
 a=q_X^*u\in H^2(X;\mathbb Z),\qquad
 A=\operatorname{PD}_X^{-1}(a)\in H_2(X;\mathbb Z).
\]

\begin{theorem}
\label{thm:equivariantsurface}
The class $A$ is invariant, $A^2=2$, and
\[
 H_2(X;\mathbb Q)^{c_*}=\mathbb Q A,\qquad K_X\cdot A=0.
\]
Every connected embedded oriented representative of $A$ has genus at least
two.  Moreover, $a$ is not of type $(1,1)$, so $A$ is not the homology class
of a complex divisor.

A smooth $+1$-sphere in $Y$, transverse to $B$, lifts to a setwise
$c$-invariant representative of $A$.  Conversely, let
$\Sigma\subset X$ be a setwise $c$-invariant connected embedded oriented
surface representing $A$.  Assume that $\Sigma$ is transverse to $X^c$.
Then $\Sigma/c$ is a smoothly embedded oriented surface
representing the positive generator of $H_2(Y;\mathbb Z)$ and has square
one.  It is a $+1$-sphere exactly when $\Sigma/c\cong S^2$.  If
$g(\Sigma)=2$, the quotient genus and number of fixed points are either
$(0,6)$ or $(1,2)$.
\end{theorem}

\begin{proof}
Lemma~\ref{lem:rationalquotient} and $b_2(Y)=1$ give
\[
 H^2(X;\mathbb Q)^{c^*}=\mathbb Q a,
\]
Since $c$ preserves the real orientation of $X$, Poincar\'e duality is
$c$-equivariant and gives the corresponding assertion in homology.
Lemma~\ref{lem:branchedthom} gives $A^2=2$.  Antiholomorphicity makes
$K_X$ anti-invariant, so $K_X\cdot A=0$.  A genus-zero representative is
excluded by the Seiberg--Witten sphere obstruction used in the proof of
Theorem~\ref{thm:spherebranch}; its class is non-torsion because $A^2=2$.
For positive genus, adjunction
gives
\[
 2g-2\geq A^2+|K_X\cdot A|=2,
\]
and therefore $g\geq2$.

Starting with a K\"ahler form $\omega$, put
\[
 \omega_-=\frac{\omega-c^*\omega}{2}.
\]
Both $\omega$ and $-c^*\omega$ are K\"ahler, so $\omega_-$ is K\"ahler and
$c^*\omega_-=-\omega_-$.  Since $A$ is invariant,
$A\cdot[\omega_-]=0$.  If $a$ were a real $(1,1)$-class, the Hodge index
theorem on the orthogonal complement of the K\"ahler class would give
$a^2<0$, contrary to $a^2=2$.

The restriction of $c$ to $\Sigma$ is orientation preserving: otherwise it
would send the non-torsion class $A$ to $-A$, although $c_*A=A$.  Thus the
transversality hypothesis gives
$T_p\Sigma\cap T_pX^c=0$ at a fixed point.  Since $T_p\Sigma$ is
$dc$-invariant, $dc|_{T_p\Sigma}=-\operatorname{id}$.  Equivariant
linearization and straightening give local coordinates
$(x,z)\in\mathbb R^2\times\mathbb C$ in which
\[
 c(x,z)=(x,-z),\qquad X^c=\{z=0\},\qquad\Sigma=\{x=0\}.
\]
In quotient coordinates $(x,Z)=(x,z^2)$, the images of $X^c$ and $\Sigma$
are $\{Z=0\}$ and $\{x=0\}$.  Hence $F=\Sigma/c$ is smooth, embedded,
oriented, and transverse to $B$, with $q_X^{-1}(F)=\Sigma$.  The branched
Thom-class identity gives
\[
 q_X^*\operatorname{PD}_Y[\Sigma/c]=a=q_X^*u.
\]
Write $\operatorname{PD}_Y[\Sigma/c]=n\,u$, since
$H^2(Y;\mathbb Z)\cong\mathbb Z u$.  Then $(n-1)a=0$; because $a^2=2$, the
class $a$ is non-torsion, so $n=1$.  Hence the quotient represents the
positive generator and has square one.
For genus two, Riemann--Hurwitz gives
\[
 -2=2(2-2h)-k,\qquad k=6-4h,
\]
where $h$ is the quotient genus and $k$ is the number of fixed points.
Thus $(h,k)=(0,6)$ or $(1,2)$.
\end{proof}

\begin{proof}[Proof of Theorem~\ref{thm:branch}]
Proposition~\ref{prop:branchmodtwo} proves the mod-two vanishing of the
branch class.  Theorem~\ref{thm:spherebranch} gives the six-point lower
bound.  The remaining assertions are contained in
Theorem~\ref{thm:equivariantsurface}.
\end{proof}

\section{Blowing down a positive sphere}

Suppose that the positive generator is represented by a smoothly embedded
sphere $E\subset Y$ with $E^2=1$.  Define
\[
 Q_E=\bigl(Y\setminus\operatorname{int}\nu E\bigr)\cup_{S^3}D^4.
\]

\begin{proposition}\label{prop:blowdown}
The manifold $Q_E$ has the integral homology of $S^1\times S^3$ and
\[
 \pi_1(Q_E)\cong\mathbb Z.
\]
Moreover,
\[
 Y\cong Q_E\#\CP^2.
\]
\end{proposition}

\begin{proof}
The Euler class of $\nu E$ is \(1\), so $\partial\nu E=S^3$.  Since $[E]$
is the primitive generator and $E^2=1$, the map
\[
 H_2(\nu E;\mathbb Z)\longrightarrow H_2(Y;\mathbb Z)
\]
is an isomorphism.  The Mayer--Vietoris sequence gives
\[
 H_*(Q_E;\mathbb Z)\cong H_*(S^1\times S^3;\mathbb Z).
\]
Indeed, for $Y=(Y\setminus\operatorname{int}\nu E)\cup\nu E$, the map
$H_2(\nu E)\to H_2(Y)$ is an isomorphism, while the oriented boundary map
$H_4(Y)\to H_3(S^3)$ is an isomorphism.  Hence, if
$W=Y\setminus\operatorname{int}\nu E$, then
\[
 H_2(W)=0,\qquad H_1(W)\cong\mathbb Z,\qquad H_3(W)\cong\mathbb Z.
\]
The Mayer--Vietoris sequence for $Q_E=W\cup D^4$ gives the displayed
homology groups.
Van Kampen gives $\pi_1(Q_E)\cong\pi_1(Y)\cong\mathbb Z$, because the common
boundary, the disk bundle, and the replacing four-ball are simply connected.
Reversing the blowdown recovers \(Y\).  The gluing map of the boundary
three-sphere extends over the four-ball by Cerf's theorem \cite{Cerf}, so
$Y\cong Q_E\#\CP^2$.
\end{proof}

To compute the homology of the universal cover, write
\(\Lambda=\mathbb Z[t,t^{-1}]\), with involution
$\overline t=t^{-1}$.

\begin{theorem}\label{thm:ZZhomology}
For the universal cover $\widetilde Q_E$ one has
\[
 H_0(\widetilde Q_E)\cong H_3(\widetilde Q_E)
       \cong\Lambda/(t-1),\qquad
 H_i(\widetilde Q_E)=0\quad(i\ne0,3).
\]
Consequently,
\[
 Q_E\simeq S^1\times S^3.
\]
\end{theorem}

\begin{proof}
Let \(p:\widetilde Y\to Y\) be the universal, equivalently infinite cyclic,
cover.  Put
\[
 W=Y\setminus\operatorname{int}\nu E,\qquad A=p^{-1}(W).
\]
Van Kampen applied to \(Y=W\cup\nu E\) along \(S^3\) gives
\(\pi_1(W)\cong\pi_1(Y)\), so \(A\) is connected.  Each component of
\(p^{-1}(\nu E)\) is the \(+1\) disk bundle \(N\to S^2\), and the components
form one free \(\mathbb Z\)-orbit.

By excision and Poincar\'e--Lefschetz duality,
\[
 H_i(\widetilde Y,A)\cong
 \begin{cases}
  \Lambda,&i=4,2,\\
  0,&\text{otherwise}.
 \end{cases}
\]
Proposition~\ref{prop:input} gives
\(H_2(\widetilde Y)\cong\Lambda\) with equivariant intersection form
\(\langle1\rangle\).  Choose a lift \(\widetilde E\).  If \(e\) is a basis
and \([\widetilde E]=f(t)e\), then disjointness of distinct translates and
\(\widetilde E^2=1\) give
\[
 f(t)\overline{f(t)}=1.
\]
With the convention
\[
 \lambda(x,y)=\sum_{k\in\mathbb Z}(x\cdot t^ky)t^k,
\]
the disjointness of distinct translates and $\widetilde E^2=1$ say exactly
that $\lambda([\widetilde E],[\widetilde E])=1$.  Comparing the largest and
smallest exponents in $f(t)\overline{f(t)}$ shows that $f$ has one nonzero
term, with coefficient $\pm1$.  Thus
\(f(t)=\pm t^k\).  Under the excision identification, the map
\[
 H_2(\widetilde Y)\longrightarrow H_2(\widetilde Y,A)
\]
is equivariant intersection with \([\widetilde E]\), and hence is
multiplication by a unit.  The long exact sequence of the pair contains
\[
 0\longrightarrow H_2(A)\longrightarrow H_2(\widetilde Y)
 \longrightarrow H_2(\widetilde Y,A)
 \longrightarrow H_1(A)\longrightarrow H_1(\widetilde Y)=0.
\]
Therefore
\[
 H_2(A)=H_1(A)=0.
\]

The universal cover \(\widetilde Q_E\) is obtained from \(A\) by attaching a
four-ball to every boundary three-sphere.  Excision gives
\[
 H_i(\widetilde Q_E,A)\cong
 \begin{cases}
  \Lambda,&i=4,\\
  0,&\text{otherwise}.
 \end{cases}
\]
Both \(\widetilde Y\) and \(\widetilde Q_E\) are connected noncompact
four-manifolds, so their fourth homology vanishes: an ordinary
top-dimensional cycle would determine a locally constant multiple of the
fundamental class and cannot have compact support unless the manifold is
compact.  The degree-three pair
sequences are therefore
\[
 0\longrightarrow\Lambda\xrightarrow{\delta_Y}H_3(A)
   \longrightarrow H_3(\widetilde Y)\longrightarrow0
\]
and
\[
 0\longrightarrow\Lambda\xrightarrow{\delta_Q}H_3(A)
   \longrightarrow H_3(\widetilde Q_E)\longrightarrow0.
\]
Choose one lift $\widetilde N$ of the disk bundle.  Excision identifies the
two copies of $\Lambda$ $\Lambda$-linearly using the relative
fundamental classes of the disk-bundle and four-ball components.  Orient a
chosen boundary sphere $\partial\widetilde N$ as the boundary of the lifted
disk bundle, and transport this orientation to every
$t^j\partial\widetilde N$.  Under the excision identifications, the relative
class $t^j$ for $(\widetilde Y,A)$ has boundary
$[t^j\partial\widetilde N]$, while the corresponding relative class for
$(\widetilde Q_E,A)$ has boundary
$-[t^j\partial\widetilde N]$, because the attached four-ball induces the
opposite boundary orientation.  A different initial choice of lift
multiplies both excision identifications and both boundary maps by the same
unit $t^m$.  Consequently
\[
 \operatorname{im}\delta_Q=\operatorname{im}\delta_Y
 \quad\text{in }H_3(A),
\]
equivalently the two maps differ only by multiplication by a unit
$\pm t^\ell\in\Lambda$.  Their cokernels are therefore equivariantly
isomorphic, and hence
\[
 H_3(\widetilde Q_E)\cong H_3(\widetilde Y).
\]
The homeomorphism
\(Y\approx\CP^2\#(S^1\times S^3)\) induces an automorphism of
$\pi_1\cong\mathbb Z$, necessarily $t\mapsto t^{\pm1}$.  It therefore
preserves the kernel defining the universal cover and lifts equivariantly,
after replacing $t$ by $t^{-1}$ when necessary.  The standard cover has
third homology \(\Lambda/(t-1)\): this follows by cutting along a separating
three-sphere and applying equivariant Mayer--Vietoris.  The pair sequence for
\((\widetilde Q_E,A)\), together with \(H_2(A)=H_1(A)=0\), gives
\(H_2(\widetilde Q_E)=H_1(\widetilde Q_E)=0\).  Thus
\[
 H_0(\widetilde Q_E)\cong H_3(\widetilde Q_E)
       \cong\Lambda/(t-1),
 \qquad H_i(\widetilde Q_E)=0\quad(i\ne0,3).
\]

The smooth manifolds and covers used here have the homotopy type of CW
complexes.  The space \(\widetilde Q_E\) is simply connected and has the
homology of \(S^3\).  The Hurewicz and Whitehead theorems
\cite[Theorems~4.32 and~4.5]{Hatcher} therefore give
\(\widetilde Q_E\simeq S^3\).  The deck generator acts trivially on
\(H_3\), so on an \(S^3\)-model it is a degree-one self-homotopy equivalence.
Since \([S^3,S^3]\) is classified by degree, this monodromy is homotopic to
the identity.  More formally, the free deck action gives a fibration
\[
 \widetilde Q_E\longrightarrow
 E\mathbb Z\times_{\mathbb Z}\widetilde Q_E\longrightarrow S^1.
\]
Projection to the orbit space $Q_E$ is a homotopy equivalence because its
fiber is the contractible space $E\mathbb Z$.  After choosing a homotopy
equivalence $\widetilde Q_E\simeq S^3$, the deck generator induces a
self-homotopy equivalence of $S^3$ of degree one, since it acts trivially on
$H_3(\widetilde Q_E)\cong\Lambda/(t-1)$.  The total space is therefore
homotopy equivalent to the mapping torus of the identity, namely
$S^1\times S^3$.
\end{proof}

\begin{proof}[Proof of Theorem~\ref{thm:blowdown}]
Proposition~\ref{prop:blowdown} gives the fundamental group and the connected
sum decomposition.  Theorem~\ref{thm:ZZhomology} gives the homology of
the infinite cyclic cover and the asserted homotopy equivalence.
\end{proof}

\begin{corollary}\label{cor:spheredichotomy}
If the positive generator has no smoothly embedded spherical representative,
then \(Y\not\cong\CP^2\#(S^1\times S^3)\).  If it is represented by a
smooth \(+1\)-sphere \(E\), then, for every such $E$,
\[
 Y\cong Q_E\#\CP^2,
\]
where \(Q_E\) is a \(\mathbb Z[\mathbb Z]\)-homology
\(S^1\times S^3\) and \(Q_E\simeq S^1\times S^3\).
\end{corollary}

\begin{proof}
In the standard connected sum, the positive generator is represented by the
complex line in the \(\CP^2\) summand.  The remaining assertions follow from
Proposition~\ref{prop:blowdown} and Theorem~\ref{thm:ZZhomology}.
\end{proof}

After choosing a spin structure and a homology orientation, the
Casson--Seiberg--Witten invariant of \(Q_E\) is defined \cite{MRS}.
Theorem~\ref{thm:ZZhomology} also places \(Q_E\) in the admissible class
for the Furuta--Ohta invariant \cite{FO,RS}.  Neither invariant is evaluated
here.  Since both vanish for $S^1\times S^3$, a nonzero value of either
invariant would show that this particular blowdown is not diffeomorphic to
\(S^1\times S^3\).

\section{Further constructions and questions}

\subsection{Two copies of the quotient}

Put $Z=Y_1\#Y_2$, where the connected-sum balls are disjoint from Albanese
generators $\gamma_i\subset Y_i$.  Write
$P_i^{\varepsilon_i}=(Y_i)_{\gamma_i}^{\varepsilon_i}$ for either normal
framing.

\begin{proposition}\label{prop:doublecirclesurgery}
One has
\[
 \pi_1(Z)\cong F_2,\qquad e(Z)=0,\qquad \sigma(Z)=2,
\]
and
\[
 Z\approx
 \CP^2\#\CP^2\#(S^1\times S^3)\#(S^1\times S^3).
\]
Surgery on one summand generator gives
\[
 Z_{\gamma_1}^{\varepsilon_1}
 \cong P_1^{\varepsilon_1}\#Y_2
 \approx\CP^2\#\CP^2\#(S^1\times S^3),
\]
while surgery on both disjoint generators gives, for all four framing pairs,
\[
 Z_{\gamma_1,\gamma_2}^{\varepsilon_1,\varepsilon_2}
 \cong P_1^{\varepsilon_1}\#P_2^{\varepsilon_2}
 \approx\CP^2\#\CP^2.
\]
If $Y\cong\CP^2\#(S^1\times S^3)$, then both homeomorphisms after surgery
may be replaced by diffeomorphisms.
\end{proposition}

\begin{proof}
Van Kampen gives $\pi_1(Z)\cong\mathbb Z*\mathbb Z=F_2$, and the
connected-sum formulas give $e(Z)=0$ and $\sigma(Z)=2$.  Applying
Proposition~\ref{prop:input} in each summand gives the asserted homeomorphism
type of $Z$.

The connected-sum neck and the second summand are disjoint from
$\nu\gamma_1$, so the first surgery takes place entirely in $Y_1$.  Hence
\[
 Z_{\gamma_1}^{\varepsilon_1}
 \cong (Y_1)_{\gamma_1}^{\varepsilon_1}\#Y_2.
\]
Theorem~\ref{thm:circle} identifies the first factor topologically with
$\CP^2$.  The two surgery neighborhoods lie in different summands, so the
surgeries commute and give the second displayed diffeomorphism.  Each factor
is homeomorphic to $\CP^2$.  If $Y$ is standard,
Corollary~\ref{cor:standardquotient} identifies each
$P_i^{\varepsilon_i}$ diffeomorphically with $\CP^2$.
\end{proof}

Choose fixed-point charts in the two copies with the same oriented local
branch model.  Remove invariant four-balls and glue their boundaries by an
orientation-reversing equivariant diffeomorphism whose restriction to the
fixed equators also reverses the boundary orientation.  The oriented source
is therefore $X\#X$, rather than $X\#\overline X$, and the quotient is
$Y\#Y$.  This equivariant connected sum gives a branched covering
This gives a branched covering
\[
 X\#X\longrightarrow Y\#Y
\]
with branch surface $B\#B\cong\#_6\mathbb{RP}^2$ and twisted normal Euler
number
$e_N(B\#B)=e_N(B)+e_N(B)=4$.  Indeed, choose generic local sections of the
twisted normal bundles that are constant near the deleted disks.  The
boundary values match under the chosen gluing, and the signed zero counts
add; this is precisely the twisted Euler-number pairing.  Equivalently, the
branched-cover signature formula gives
$2=2\cdot2-e_N(B\#B)/2$.  The gauge-theoretic and Hodge-theoretic arguments of Section~4
do not extend directly to $X\#X$: the connected sum has vanishing ordinary
Seiberg--Witten invariants and no inherited K\"ahler structure.  Thus the
proposition is a surgery application, not a two-copy version of the
six-point obstruction.

\subsection{Two smooth questions}

Two smooth questions remain.  The first controls whether the
topological splitting of $Y$ can be realized smoothly by a positive sphere.
The second asks whether circle surgery produces the standard $\CP^2$ and,
more strongly, whether the marked dual sphere carries nonstandard smooth
information.

In the standard model $\CP^2\#(S^1\times S^3)$, surgery on the product
generator replaces the $S^1\times S^3$ summand by $S^4$.  We call the belt
sphere in this $S^4$ summand, equipped with its inherited reverse-surgery
parametrization, the \emph{standard parametrized dual sphere}.

Nothing in the preceding calculations rules out the existence of a positive
sphere.  We suspect that such a sphere is present, although it is not visible
in the current description of the quotient.  The remaining problem is to
construct it and determine the smooth type of its blowdown.

\begin{question}\label{q:blowdown}
Is the positive generator of $H_2(Y;\mathbb Z)$ represented by a smoothly
embedded $+1$-sphere?  If so, is the corresponding blowdown diffeomorphic to
$S^1\times S^3$?
\end{question}

\begin{question}\label{q:one}
For either normal framing of an Albanese generator $\gamma$, is
$Y_\gamma$ diffeomorphic to $\CP^2$?  More strongly, if it is standard, is
its parametrized dual sphere equivalent to the standard parametrized dual
sphere?
\end{question}

The authors plan to study these remaining questions in future work.

\appendix
\section{Exact computational certificates}\label{app:certificates}
\small

This appendix records the computer-assisted parts of the proofs and explains
how they may be checked independently.  The accompanying ancillary archive
contains the input data, exact certificates, verification scripts, output
logs, software requirements, and SHA256 checksums.  The final verifications
use integer, finite-field, or exact cyclotomic arithmetic.  The bounded
searches retained in the archive were used only to find candidates and are
not part of the proofs.  The complete ancillary files are available at
\url{https://github.com/akhmedov-coder/cartwright-steger-real-albanese}.

\subsection{Source dictionary and master verification}

The group presentation and the matrices for $u,v,b$ are those of
\cite[Theorem~1.1 and the formulas preceding it]{CKY}.  The words
$a_1,a_2,a_3$, the subgroup of index $864$, and the integral homology
coordinates are from \cite[Theorem~1.6]{CKY}.  The reflection matrix $K$
and the words describing its action are taken from
\cite[Section~2.3, especially Proposition~2.3]{SS}; the reduction to the
finite field follows the integral model in \cite[Section~1.5]{CKY}.  For
the weight-two calculation, the archived Borisov--Yeung source is the
source distribution at \url{https://arxiv.org/src/1804.00737}, and the
archived equation of $F_1$ is from
\url{https://arxiv.org/src/1903.03017}.  The archive records exact source
filenames and SHA256 hashes, so each hard-coded datum can be compared with
its cited source.

After installing the versions listed in \texttt{requirements.txt}, all
proof certificates are checked from the extracted \texttt{supplement}
directory by
\begin{verbatim}
python3 run_all.py
\end{verbatim}
The top-level \texttt{SHA256SUMS} covers the source copies, scripts,
certificates, and logs.  The individual commands are described below.

\subsection{The fixed word and its homology class}

Proposition~\ref{prop:k-one-two} uses the word
\[
 W=u^{-1}v^{-1}uv^2u^{-1}vbu^{-1}v^{-1}uv^{-1}b^{-1}uv^{-1}.
\]
The file \texttt{cky\_coset\_table.pkl} contains the $864$-row right-coset
table for $\Pi$ in the ambient CKY group.  The script
\texttt{verify\_w1\_exact.py} reads the word in this table and checks
$W\in\Pi$.  It then verifies
\[
                         K\overline W K^{-1}=W
\]
in the projective matrix model over $\mathbb Q(\zeta_{12})$.

The files \texttt{cocycle\_first.txt} and
\texttt{cocycle\_second.txt} contain two integer cellular one-cocycles on
the $864$-sheeted presentation complex.  The script
\texttt{verify\_certificate.py} checks every entry of the coset table,
every lifted relator, the normalizations on $a_1,a_2,a_3$, and the values of
the two cocycles on $W$.  The resulting pair is $(2,-2)$.  These checks are
independent of the numerical procedure used to find the cocycles.

The first three fixed-word checks run as part of \texttt{run\_all.py}.
They may instead be invoked individually as \texttt{sympy\_cosets.py},
\texttt{verify\_w1\_exact.py}, and \texttt{verify\_certificate.py}.
The cocycles may also be regenerated by the optional script
\path{cover_cocycles.py}.  A second optional script,
\path{search_ambient_centralizer.py}, reproduces the bounded search for
fixed words.  Its output is not used to certify the proposition.

\subsection{The finite-field obstruction}

The script \texttt{f4\_fixed\_obstruction.py} supplies the exact finite
calculation used in Theorem~\ref{thm:finite-obstruction}.  It implements
$\mathbb F_4$ arithmetic directly, verifies the reduced CKY relators and the
three formulas for the reflection on the generators, enumerates the group
$Q$, and checks the induced involution on every element.  Its output is
\begin{verbatim}
quotient_size 72
matrix_image_size 72
kernel_homology_labels [(0, 0)]
fixed_size 6
fixed_homology Counter({(0, 0): 6})
odd_fixed_exists False
\end{verbatim}
Thus the six fixed elements all have zero mod-two homology label.  This is
the finite assertion used in the proof that
$\operatorname{ab}(\Fix\phi)=2\mathbb Z(1,-1)$ and hence that the oval
number is two.  The script uses only Python~3 and no floating-point
arithmetic.

\subsection{The weight-two pencil}

The weight-two certificate contains a parser-ready transcription of the
$85$ Borisov-Yeung equations and the characteristic-zero equations for the
fiber $F_1$ from the ancillary files to \cite{Rito}.  The archive also
contains byte-for-byte copies of the source files.  The scripts
\texttt{verify\_source\_provenance.py} and
\texttt{verify\_coordinate\_points.py} compare the transcribed equations
with their sources and check that $[U_7],[U_8],[U_9]$ lie on $X$.

With Python~3 and SymPy~1.14, the projective calculations used in
Theorem~\ref{thm:weight-two-pencil} are reproduced by
\begin{verbatim}
python3 test_weight2_base.py 10007
python3 test_weight2_base.py 10009
python3 test_weight2_base.py 10037
python3 test_weight2_pencil.py 10007
python3 test_weight2_pencil.py 10009
python3 test_weight2_pencil.py 10037
\end{verbatim}
For the first three runs, each of the seven affine charts must return
\texttt{empty=True}.  For the last three runs, the common base scheme must
return \texttt{empty=True}.  The archived logs record complete runs at all
three primes.  The file \texttt{SHA256SUMS} records the hashes of the source
files, scripts, and logs.

\bigskip
\begin{center}
 {\sc AI Use Disclosure} 
 \end{center}
 
 \medskip
 LLM tools were used for assistance for checking grammar and language editing. 
 All mathematical ideas, arguments, and results in this paper are due to the authors, who take full responsibility for the content.

\footnotesize
\sloppy
\hbadness=10000

\end{document}